\documentclass[12pt]{article}
\usepackage{amsmath,amssymb,amsthm}
\numberwithin{equation}{section}
\usepackage{hyperref}
\usepackage{units}
\usepackage{color}
\usepackage[T1]{fontenc}
\usepackage[utf8]{inputenc}
\usepackage{authblk}
\usepackage{bm}
\usepackage{graphicx}

\usepackage{datetime}
\usepackage{color}

\usepackage[toc,page]{appendix}

\newcommand{\Nzero}{{\mathbb N}_{0}}
\newcommand{\Zplus}{{\mathbb Z^{+}}}

\newtheorem{thm}{Theorem}[section]

\newtheorem{lemma}[thm]{Lemma}

\newtheorem{cor}[thm]{Corollary}

\title{Generalizations for reciprocal Fibonacci-Lucas sums of Brousseau\thanks{AMS Classification Numbers : 11B37, 11B39}\vspace{10mm}} 

\author[]{Kunle Adegoke \thanks{kunle.adegoke@yandex.com, adegoke00@gmail.com}}

\affil{Department of Physics and Engineering Physics, \mbox{Obafemi Awolowo University}, Ile-Ife, Nigeria}

\begin{document}

\date{}

\maketitle

\begin{abstract}
\noindent We derive closed form expressions for finite and infinite \mbox{Fibonacci-Lucas} sums having products of Fibonacci or Lucas numbers in the denominator of the summand. Our results generalize and extend those obtained by pioneer Brother Alfred Brousseau and later researchers.
\end{abstract}

\tableofcontents
\section{Introduction}

The Fibonacci numbers, $F_n$, and Lucas numbers, $L_n$, are defined, for $n\in\Nzero$, as usual, through the recurrence relations \mbox{$F_n=F_{n-1}+F_{n-2}$}, with $F_0=0$, $F_1=1$ and \mbox{$L_n=L_{n-1}+L_{n-2}$}, with $L_0=2$, $L_1=1$.

\bigskip

Our main aim in this paper is to derive closed form expressions for the following sums and their corresponding alternating versions, for positive integers $m$, $n$ and $q$:
\[
\sum_{k = 1}^\infty  {\frac{{L_{nk + nq} L_{nk + 2nq}  \cdots L_{nk + (m - 1)nq} }}{{F_{nk} F_{nk + nq}  \cdots F_{nk + mnq} }}} ,\quad\sum_{k = 1}^\infty  {\frac{{F_{nk + nq} F_{nk + 2nq}  \cdots F_{nk + (m - 1)nq} }}{{L_{nk} L_{nk + nq}  \cdots L_{nk + mnq} }}}\,,\quad m>1\,, 
\]
\[
\sum_{k = 1}^\infty  {\frac{1}{{F_{nk} F_{nk + nq}  \cdots F_{nk + mnq - nq} F_{nk + mnq + nq}  \cdots F_{nk + mnq + 2mnq} }}}\,, 
\]
\[
\sum_{k = 1}^\infty  {\frac{1}{{L_{nk} L_{nk + nq}  \cdots L_{nk + mnq - nq} L_{nk + mnq + nq}  \cdots L_{nk + mnq + 2mnq} }}}\,, 
\]
\[
\sum_{k = 1}^\infty  {\frac{{L_{nk + mnq} }}{{F_{nk} F_{nk + nq}  \cdots F_{nk + 2mnq} }}} ,\quad\sum_{k = 1}^\infty  {\frac{{F_{nk + mnq} }}{{L_{nk} L_{nk + nq}  \cdots L_{nk + 2mnq} }}}\,, 
\]
\[
\sum_{k = 1}^\infty  {\frac{{F_{2nk + mnq} }}{{F_{nk}^2 F_{nk + nq}^2 F_{nk + 2nq}^2  \cdots F_{nk + mnq}^2 }}} ,\quad\sum_{k = 1}^\infty  {\frac{{F_{2nk + mnq} }}{{L_{nk}^2 L_{nk + nq}^2 L_{nk + 2nq}^2  \cdots L_{nk + mnq}^2 }}}\,, 
\]
\[
\sum_{k = 1}^\infty  {\frac{{L_{nk + mnq} }}{{F_{nk}^2 F_{nk + nq}^2  \cdots F_{nk + (m - 1)nq}^2 F_{nk + mnq} F_{nk + (m + 1)q}^2  \cdots F_{nk + 2mnq}^2 }}}\,, 
\]
\[
\sum_{k = 1}^\infty  {\frac{{F_{nk + mnq} }}{{L_{nk}^2 L_{nk + nq}^2  \cdots L_{nk + (m - 1)nq}^2 L_{nk + mnq} L_{nk + (m + 1)q}^2  \cdots L_{nk + 2mnq}^2 }}}\,, 
\]
\[
\sum_{k = 1}^\infty  {\frac{{F_{2nk + mnq} L_{nk + nq}^2 L_{nk + 2nq}^2  \cdots L_{nk + (m - 1)nq}^2 }}{{F_{nk}^2 F_{nk + nq}^2 F_{nk + 2nq}^2  \cdots F_{nk + mnq}^2 }}}\,,\quad m>1\,.
\]

\bigskip

We require the following telescoping summation identities (see~\cite{adegoke})
\begin{equation}\label{equ.p2fefzx}
\sum_{k = 1}^N {\left[ {f(k) - f(k + q)} \right]}  = \sum_{k = 1}^q {f(k)}  - \sum_{k = 1}^q {f(k + N)},\quad\mbox{for $N\ge q\in\Nzero$}
\end{equation}
and
\begin{equation}\label{equ.mdgx80r}
\begin{split}
&\sum_{k = 1}^N {( - 1)^{k - 1} \left[ {f(k) + ( - 1)^{q - 1} f(k + q)} \right]}\\ 
&\quad= \sum_{k = 1}^q {( - 1)^{k - 1} f(k)}  + ( - 1)^{N - 1} \sum_{k = 1}^q {( - 1)^{k - 1} f(k + N)}\,. 
\end{split}
\end{equation}

In general, infinite sums are evaluated using
\begin{equation}\label{equ.piqoita}
\begin{split}
&\sum_{k = 1}^\infty  {\left[ {f(k) - f(k + q)} \right]},\qquad q\in\Nzero\\
&\qquad= \sum_{k = 1}^q {f(k)}  - \sum_{k = 1}^q {\mathop {\lim }_{N \to \infty } f(k + N)}
\end{split}
\end{equation}
and
\begin{equation}\label{equ.jcfrlss}
\sum_{k = 1}^\infty  {( - 1)^{k - 1} \left[ {f(k) \mp f(k + q)} \right]}  = \sum_{k = 1}^q {( - 1)^{k - 1} f(k)}\,, 
\end{equation}
where the upper sign is taken if $q$ is even and the lower if $q$ is odd.

\bigskip

If $f(N)$ approaches zero as $N$ approaches infinity, then we have, from~\eqref{equ.p2fefzx} and~\eqref{equ.mdgx80r}, the useful identities
\begin{equation}\label{equ.gy1asjs}
\sum_{k = 1}^\infty  {\left[ {f(k) - f(k + q)} \right]}  = \sum_{k = 1}^q {f(k)},\quad\mbox{$q\in\Nzero$}\,, 
\end{equation}
\begin{equation}
\sum_{k = 1}^\infty  {( - 1)^{k - 1} \left[ {f(k) \mp f(k + q)} \right]}  = \sum_{k = 1}^q {( - 1)^{k - 1} f(k)}\,,
\end{equation}
where the upper sign applies if $q$ is even and the lower if $q$ is odd.

\bigskip

The golden ratio, having the numerical value of $(\sqrt 5+1)/2$, is denoted in this paper by $\phi$.

\bigskip

We shall require the following identities (most of which can be found in the book by Vajda~\cite{vajda}):
\begin{subequations}\label{equ.d8ulygu}
\begin{eqnarray}
L_vF_u &=& F_{u + v} + ( - 1)^vF_{u-v}\label{equ.mzt8c69} \\
F_vL_u &=& F_{u + v} - (-1)^vF_{u-v}\label{equ.yb05ue2} 
\end{eqnarray}
\end{subequations}
\begin{subequations}\label{equ.c7qhgv4}
\begin{eqnarray}
2F_{u + v} &=& L_vF_u + L_uF_v \label{equ.ejrnkwy}\\ 
(-1)^v2F_{u-v} &=& F_uL_v - L_uF_v\label{equ.moytk3x}
\end{eqnarray}
\end{subequations}
\begin{subequations}\label{equ.epvyp3u}
\begin{eqnarray}
L_vL_u &=& L_{u + v} + (-1)^vL_{u-v}\label{equ.ztsb3uk}\\ 
5F_vF_u &=& L_{u + v} - (-1)^vL_{u-v}\label{equ.q79u3q4}  
\end{eqnarray}
\end{subequations}

\begin{equation}\label{equ.yr8t5vs}
(-1)^{u-1}(F_{v+u} F_{v-u})=F_u^2(F_{v+1}F_{v-1}) - F_v^2(F_{u+1} F_{ u-1})
\end{equation}

\begin{subequations}\label{equ.hn7lvbv}
\begin{eqnarray}
( -1)^tF_u F_v=F_{t+u}F_{t+v} - F_tF_{t+u+v}\label{equ.tokbcvq}\\ 
( -1)^{t+1}5F_u F_v=L_{t+u}L_{t+v} - L_tL_{t+u+v} \label{equ.owip3z8}  
\end{eqnarray}
\end{subequations}

\begin{subequations}\label{equ.kp0k485}
\begin{eqnarray}
F_{u-v}F_{u+v} &=& F_u^2 + (-1)^{u+v-1}F_v^2\label{equ.ded0k7c}\\ 
5F_{u-v}F_{u+v} &=& L_u^2 + (-1)^{u+v-1}L_v^2\label{equ.wodrq78}  
\end{eqnarray}
\end{subequations}

\begin{subequations}\label{equ.loghvyf}
\begin{eqnarray}
F_v F_{2u + v + p} &=& F_{u + v + p} F_{u + v}  + ( - 1)^{v + 1} F_{u + p} F_u  \label{equ.cqtsjoj}\\ 
F_v L_{2u + v + p} &=& L_{u + v + p} F_{u + v}  + ( - 1)^{v + 1} L_{u + p} F_u \label{equ.dt4czzk}  
\end{eqnarray}
\end{subequations}

The identities~\eqref{equ.kp0k485} and~\eqref{equ.loghvyf} were proved by Howard in~\cite{howard}.
 
\bigskip

The following limiting values are readily established using Binet's formula:
\begin{subequations}\label{equ.y4ibpqi}
\begin{eqnarray}
\mathop {\lim }_{N \to \infty } \frac{{F_{N + m} }}{{F_{N + n} }} &=& \phi ^{m - n}  = \mathop {\lim }_{N \to \infty } \frac{{L_{N + m} }}{{L_{N + n} }} ,\\
\mathop {\lim }_{N \to \infty } \frac{{F_{N + m} }}{{L_{N + n} }} &=& \frac{{\phi ^{m - n} }}{{\sqrt 5 }} = \frac{1}{5}\mathop {\lim }_{N \to \infty } \frac{{L_{N + m} }}{{F_{N + n} }}\,.  
\end{eqnarray}
\end{subequations}

\bigskip

We shall adopt the following conventions for empty sums and empty products:
\[
\sum_{k = 1}^0 {f(k)}  = 0,\quad\prod_{k = 1}^0 {f(k)}  = 1\,.
\]

\section{Results: Generalizations}

\subsection{Telescoping summation identities}
\begin{lemma}\label{finall}
If $m$, $n$, $q$ and $N$ are positive integers and $f(k)$ is a real sequence, then
\[
\begin{split}
&\sum_{k = 1}^N {\left\{ {\left[ {f(nk) - f(nk + mnq)} \right]\prod_{j = 1}^{m - 1} {f(nk + jnq)} } \right\}}\\
&\qquad= \sum_{k = 1}^{q} {\left\{ {\prod_{j = 0}^{m - 1} {f(nk + jnq)} } \right\}}  - \sum_{k = 1}^{q} {\left\{ {\prod_{j = 0}^{m - 1} {f(nk + nN + jnq)} } \right\}}\,.
\end{split}
\]
\end{lemma}
If the sequence $f(k)$ is convergent and we denote by $f_\infty$ the limiting value of $f(nN)$ as $N$ approaches infinity, we have
\begin{equation}\label{infall}
\begin{split}
&\sum_{k = 1}^\infty  {\left\{ {\left[ {f(nk) - f(nk + mnq)} \right]\prod_{j = 1}^{m - 1} {f(nk + jnq)} } \right\}}\\
&\qquad= \sum_{k = 1}^{q} {\left\{ {\prod_{j = 0}^{m - 1} {f(nk + jnq)} } \right\}}  - f_\infty{}^mq\,.
\end{split}
\end{equation} 
\begin{proof}
We have
\[
\begin{split}
&\left[ {f(nk) - f(nk + mnq)} \right]\prod_{j = 1}^{m - 1} {f(nk + jnq)}\\
&\qquad= f(nk)\prod_{j = 1}^{m - 1} {f(nk + jnq)}  - f(nk + mnq)\prod_{j = 1}^{m - 1} {f(nk + jnq)}\\
&\qquad = \prod_{j = 0}^{m - 1} {f(nk + jnq)}  - \prod_{j = 1}^{m} {f(nk + jnq)}
\end{split}
\]
\begin{equation}\label{equ.tejlmvh}
\begin{split}
&\qquad= \prod_{j = 0}^{m - 1} {f(nk + jnq)}  - \prod_{j = 0}^{m - 1} {f(nk + jnq + nq)}\\
&\qquad\equiv \prod_{j = 0}^{m - 1} {f(nk + jnq)}  - \left. {\prod_{j = 0}^{m - 1} {f(nk + jnq)} } \right|_{k \to k + q}\,.
\end{split}
\end{equation}
The result follows by summing both sides of identity~\eqref{equ.tejlmvh}, using the identity~\eqref{equ.p2fefzx} to perform the telescopic summation on the right hand side.

\end{proof}
\begin{lemma}\label{finqeven}
If $f(k)$ is a real sequence and $m$, $n$, $q$ and $N$ are positive integers such that $q$ is even, then
\[
\begin{split}
&\sum_{k = 1}^N {( - 1)^{k - 1} \left\{ {\left[ {f(nk) - f(nk + mnq)} \right]\prod_{j = 1}^{m - 1} {f(nk + jnq)} } \right\}}\\
&\quad= \sum_{k = 1}^q {( - 1)^{k - 1} \left\{ {\prod_{j = 0}^{m - 1} {f(nk + jnq)} } \right\}}\\
&\qquad + ( - 1)^{N - 1} \sum_{k = 1}^q {( - 1)^{k - 1} \left\{ {\prod_{j = 0}^{m - 1} {f(nk + nN + jnq)} } \right\}}\,. 
\end{split}
\]
\end{lemma}
\begin{proof}
Multiply through the identity~\eqref{equ.tejlmvh} by $(-1)^{k-1}$ and use identity~\eqref{equ.mdgx80r}.
\end{proof}
\begin{lemma}\label{finqodd}
If $f(k)$ is a real sequence and $m$, $n$, $q$ and $N$ are positive integers such that $q$ is odd, then
\[
\begin{split}
&\sum_{k = 1}^N {( - 1)^{k - 1} \left\{ {\left[ {f(nk) + f(nk + mnq)} \right]\prod_{j = 1}^{m - 1} {f(nk + jnq)} } \right\}}\\
&\quad= \sum_{k = 1}^q {( - 1)^{k - 1} \left\{ {\prod_{j = 0}^{m - 1} {f(nk + jnq)} } \right\}}\\
&\qquad + ( - 1)^{N - 1} \sum_{k = 1}^q {( - 1)^{k - 1} \left\{ {\prod_{j = 0}^{m - 1} {f(nk + nN + jnq)} } \right\}}\,. 
\end{split}
\]

\end{lemma}
\begin{proof}
We have the identity
\begin{equation}
\begin{split}
&\left[ {f(nk) + f(nk + mnq)} \right]\prod_{j = 1}^{m - 1} {f(nk + jnq)}\\
&\qquad= \prod_{j = 0}^{m - 1} {f(nk + jnq)}  + \prod_{j = 0}^{m - 1} {f(nk + jnq + nq)}\\
&\qquad\equiv \prod_{j = 0}^{m - 1} {f(nk + jnq)}  + \left. {\prod_{j = 0}^{m - 1} {f(nk + jnq)} } \right|_{k \to k + q}\,,
\end{split}
\end{equation}
from which the result follows after multiplying through by $(-1)^{k-1}$ and summing over $k$, making use of the identity~\eqref{equ.mdgx80r}.
\end{proof}
If the sequence $f(k)$ is convergent and $f(2Nn)$ and $f((2N-1)n)$ both have the same limiting value as $N$ approaches infinity, we have
\begin{equation}\label{infallalt}
\begin{split}
&\sum_{k = 1}^\infty {( - 1)^{k - 1} \left\{ {\left[ {f(nk) \mp f(nk + mnq)} \right]\prod_{j = 1}^{m - 1} {f(nk + jnq)} } \right\}}\\
&\quad= \sum_{k = 1}^q {( - 1)^{k - 1} \left\{ {\prod_{j = 0}^{m - 1} {f(nk + jnq)} } \right\}}\,,
\end{split}
\end{equation}
where the upper sign is taken if $q$ is even and the lower if $q$ is odd.

\subsection{Sums with $F_{nk}F_{nk+nq}\cdots F_{nk+mnq}$ or \\\mbox{$F_{nk}F_{nk+nq}\cdots F_{nk+(m-1)nq}F_{nk+(m+1)nq}\cdots F_{nk+2mnq}$} in the denominator}\label{sec.u8lgixm}
\begin{thm}\label{thm.vn8ph53}
If $m$, $n$ and $q$ are positive integers, then
\[
\sum_{k = 1}^\infty  {\left[ {( - 1)^{nk - 1} \frac{{\prod_{j = 1}^{m - 1} {L_{nk + jnq} } }}{{\prod_{j = 0}^m {F_{nk + jnq} } }}} \right]}  = \frac{{q\sqrt{ 5^m} }}{{2F_{mnq} }} - \frac{1}{{2F_{mnq} }}\sum_{k = 1}^q {\prod_{j = 0}^{m - 1} {\frac{{L_{nk + jnq} }}{{F_{nk + jnq} }}} }\,, 
\]
so that
\begin{equation}
\sum_{k = 1}^\infty  {\left[ {\frac{{\prod_{j = 1}^{m - 1} {L_{nk + jnq} } }}{{\prod_{j = 0}^m {F_{nk + jnq} } }}} \right]}  = \frac{1}{{2F_{mnq} }}\sum_{k = 1}^q {\prod_{j = 0}^{m - 1} {\frac{{L_{nk + jnq} }}{{F_{nk + jnq} }}} }  - \frac{{q\sqrt {5^m } }}{{2F_{mnq} }},\quad\mbox{$n$ even}
\end{equation}
and
\begin{equation}
\sum_{k = 1}^\infty  {\left[ {( - 1)^{k - 1} \frac{{\prod_{j = 1}^{m - 1} {L_{nk + jnq} } }}{{\prod_{j = 0}^m {F_{nk + jnq} } }}} \right]}  = \frac{{q\sqrt {5^m } }}{{2F_{mnq} }} - \frac{1}{{2F_{mnq} }}\sum_{k = 1}^q {\prod_{j = 0}^{m - 1} {\frac{{L_{nk + jnq} }}{{F_{nk + jnq} }}} } ,\quad\mbox{$n$ odd}\,.
\end{equation}
\end{thm}
In particular,
\begin{equation}\label{equ.xk94y11}
\sum_{k = 1}^\infty  {\frac{{( - 1)^{nk - 1} }}{{F_{nk} F_{nk + nq} }}}  = \frac{{q\sqrt 5 }}{{2F_{nq} }} - \frac{1}{{2F_{nq} }}\sum_{k = 1}^q {\frac{{L_{nk} }}{{F_{nk} }}}\,.
\end{equation}
Brousseau's result (equation~(3) of~\cite{brousseau2}, also rederived in various equivalent forms by other authors, see for example reference~\cite{rabinowitz}) corresponds to setting $n=1$ in~\eqref{equ.xk94y11}, but with a different, but equivalent, form for the right hand side. Bruckman and Good's result (equation~(19) of~\cite{bruckman}) is also a special case of~\eqref{equ.xk94y11}, corresponding to setting $q=1$.
\begin{thm}\label{thm.hxvgf6w}
If $m$, $n$ and $q$ are integers such that $n$ is odd and $q$ is even, then
\[
\sum_{k = 1}^\infty  {\left[ {\frac{{\prod_{j = 1}^{m - 1} {L_{nk + jnq} } }}{{\prod_{j = 0}^m {F_{nk + jnq} } }}} \right]}  = \frac{1}{{2F_{mnq} }}\sum_{k = 1}^q {\left[ {( - 1)^k \prod_{j = 0}^{m - 1} {\frac{{L_{nk + jnq} }}{{F_{nk + jnq} }}} } \right]}\,. 
\]
\end{thm}
In particular,
\begin{equation}\label{equ.lldx63d}
\sum_{k = 1}^\infty  {\frac{1}{{F_{nk} F_{nk + nq} }} = \frac{1}{{2F_{nq} }}\sum_{k = 1}^q {\left[ {( - 1)^k \frac{{L_{nk} }}{{F_{nk} }}} \right]} },\quad\mbox{$n$ odd, $q$ even}\,,
\end{equation}
which generalizes the result obtained by Rabinowitz (the second of equation~(26) of~\cite{rabinowitz}), the latter corresponding to the special case $n=1$ in the identity~\eqref{equ.lldx63d}, but with a different, but equivalent, form for the right hand side.
\subsubsection*{Proof of Theorem~\ref{thm.vn8ph53} and Theorem~\ref{thm.hxvgf6w}}
Dividing through the identity~\eqref{equ.moytk3x} by $F_uF_v$ and setting $u=nk+mnq$ and $v=nk$, the following identity is established for $k$, $m$, $n$ and $q$ positive integers:
\[
\frac{{( - 1)^{nk - 1} 2F_{mnq} }}{{F_{nk} F_{nk + mnq} }} = \frac{{L_{nk + mnq} }}{{F_{nk + mnq} }} - \frac{{L_{nk} }}{{F_{nk} }}\,.
\]
Using $f(k)=L_k/F_k$ in identity~\eqref{finall} we get the finite summation identity
\begin{equation}\label{equ.oy4215f}
\begin{split}
&2F_{mnq} \sum_{k = 1}^N {\left[ {( - 1)^{nk - 1} \frac{{\prod_{j = 1}^{m - 1} {L_{nk + jnq} } }}{{\prod_{j = 0}^m {F_{nk + jnq} } }}} \right]}\\
&\quad  = \sum_{k = 1}^q {\left[ {\prod_{j = 0}^{m - 1} {\frac{{L_{nk + nN + jnq} }}{{F_{nk + nN + jnq} }}} } \right]}  - \sum_{k = 1}^q {\left[ {\prod_{j = 0}^{m - 1} {\frac{{L_{nk + jnq} }}{{F_{nk + jnq} }}} } \right]}\,,
\end{split}
\end{equation}
which yields Theorem~\ref{thm.vn8ph53} in the limit $N$ approaches infinity. Theorem~\ref{thm.hxvgf6w} is proved by using $f(k)=L_k/F_k$ in identity~\eqref{infallalt}.

\bigskip

\begin{thm}\label{thm.xd29ih0}
If $m$, $n$ and $q$ are positive \underline{odd} integers, then
\[
\sum_{k = 1}^\infty  {\left[ {\frac{1}{{\prod_{j = 0}^{m - 1} {F_{nk + njq} } \prod_{j = m + 1}^{2m} {F_{nk + njq} } }}} \right]}  = \frac{1}{{L_{mnq} }}\sum_{k = 1}^q {\left[ {\frac{1}{{\prod_{j = 0}^{2m - 1} {F_{nk + njq} } }}} \right]}\,. 
\]

\end{thm}
Below are a few explicit examples from Theorem~\ref{thm.xd29ih0}:
\begin{equation}
\begin{split}
&\mbox{At $m=1$}:\\
&\sum_{k = 1}^\infty  {\frac{1}{{F_{nk} F_{nk + 2nq} }} = \frac{1}{{L_{nq} }}\sum_{k = 1}^q {\frac{1}{{F_{nk} F_{nk + nq} }}} },\quad\mbox{$nq$ odd}\,.
\end{split}
\end{equation}
\[
\begin{split}
&\mbox{At $(m,n,q)=(1,1,1)$ and $(m,n,q)=(1,1,3)$:}\\
&\sum_{k = 1}^\infty  {\frac{1}{{F_k F_{k + 2} }}}  = 1,\quad \sum_{k = 1}^\infty  {\frac{1}{{F_k F_{k + 6} }}}  =\frac{143}{960}\,,
\end{split}
\]
corresponding to Formulas~(4) and~(6) of Brousseau in reference~\cite{brousseau1}.
\[
\begin{split}
&\mbox{At $(m,n,q)=(3,1,3)$:}\\
&\sum_{k = 1}^\infty  {\frac{1}{{F_k F_{k + 3} F_{k + 6} F_{k + 12} F_{k + 15} F_{k + 18} }}}  = \frac{{{\rm 938359017897442612}}}{{{\rm 5579104720519492358676480}}}\,.
\end{split}
\]
\[
\begin{split}
&\mbox{At $(m,n,q)=(5,3,1)$:}\\
&\sum_{k = 1}^\infty  {\frac{1}{{F_{3k} F_{3k + 3} F_{3k + 6} F_{3k + 9} F_{3k + 12} F_{3k + 18} F_{3k + 21} F_{3k + 24} F_{3k + 27} F_{3k + 30} }}}\\
&\qquad= \frac{{\rm 1}}{{{\rm 13970032097862115517068710877593600}}}\,.
\end{split}
\]
\begin{thm}\label{thm.i3apd20}
If $q$, $m$ and $n$ are positive integers such that $q$ is odd and $nm$ is even, then
\[
\begin{split}
\sum_{k = 1}^\infty {\left[ {\frac{{( - 1)^{k - 1} }}{{\prod_{j = 0}^{m - 1} {F_{nk + njq} } \prod_{j = m + 1}^{2m} {F_{nk + njq} } }}} \right]}  &= \frac{1}{{L_{mnq} }}\sum_{k = 1}^q {\left[ {\frac{{( - 1)^{k - 1} }}{{\prod_{j = 0}^{2m - 1} {F_{nk + njq} } }}} \right]}\,.
\end{split}
\]

\end{thm}
Examples from Theorem~\ref{thm.i3apd20} include
\begin{equation}
\begin{split}
&\mbox{At $m=2$}:\\
&\sum_{k = 1}^\infty  {\frac{{( - 1)^{k - 1} }}{{F_{nk} F_{nk + nq} F_{nk + 3nq} F_{nk + 4nq} }} = \frac{1}{{L_{2nq} }}\sum_{k = 1}^q {\frac{{( - 1)^{k - 1} }}{{F_{nk} F_{nk + nq} F_{nk + 2nq} F_{nk + 3nq} }}} },\quad\mbox{$q$ odd}\,.
\end{split}
\end{equation}
\[
\begin{split}
&\mbox{At $(m,n,q)=(1,2,1)$ and $(m,n,q)=(2,1,1)$:}\\
&\sum_{k = 1}^\infty  {\frac{{( - 1)^{k - 1} }}{{F_{2k} F_{2k + 4} }}}  = \frac{1}{9},\quad\sum_{k = 1}^\infty  {\frac{{( - 1)^{k - 1} }}{{F_k F_{k + 1} F_{k + 3} F_{k + 4} }}}  = \frac{1}{18}\,.
\end{split}
\]
\[
\begin{split}
&\mbox{At $(m,n,q)=(2,6,1)$:}\\
&\sum_{k = 1}^\infty  {\frac{{( - 1)^{k - 1} }}{{F_{6k} F_{6k + 6} F_{6k + 18} F_{6k + 24} }}}  = \frac{1}{{{\rm 44444622716928}}}\,.
\end{split}
\]

\subsubsection*{Proof of Theorem~\ref{thm.xd29ih0} and Theorem~\ref{thm.i3apd20}}
With $v=mnq$ and $u=nk+mnq$ in identity~\eqref{equ.mzt8c69}, the following identity is established:
\begin{equation}
\frac{{L_{mnq} F_{nk + mnq} }}{{F_{nk} F_{nk + 2mnq} }} = \frac{1}{{F_{nk} }} + \frac{{( - 1)^{mnq} }}{{F_{nk + 2mnq} }}\,,
\end{equation}
so that
\begin{equation}\label{equ.eo7cizi}
\frac{{L_{mnq} F_{nk + mnq} }}{{F_{nk} F_{nk + 2mnq} }} = \frac{1}{{F_{nk} }} - \frac{1}{{F_{nk + 2mnq} }},\quad\mbox{$mnq$ odd}
\end{equation}
and
\begin{equation}\label{equ.w1u72rp}
\frac{{L_{mnq} F_{nk + mnq} }}{{F_{nk} F_{nk + 2mnq} }} = \frac{1}{{F_{nk} }} + \frac{1}{{F_{nk + 2mnq} }},\quad\mbox{$mnq$ even}\,.
\end{equation}
If $m$, $n$ and $q$ are positive \underline{odd} integers, then from~\eqref{equ.eo7cizi} and using $f(k)=1/F_k$ in Lemma~\ref{finall} (with $m\to 2m$), we have the following definite summation identity
\begin{equation}
\begin{split}
&\sum_{k = 1}^N {\left[ {\frac{1}{{\prod_{j = 0}^{m - 1} {F_{nk + jnq} } \prod_{j = m + 1}^{2m} {F_{nk + jnq} } }}} \right]}\\
&\qquad = \frac{1}{{L_{mnq} }}\sum_{k = 1}^q {\left[ {\frac{1}{{\prod_{j = 0}^{2m - 1} {F_{nk + jnq} } }}} \right]}  - \frac{1}{{L_{mnq} }}\sum_{k = 1}^q {\left[ {\frac{1}{{\prod_{j = 0}^{2m - 1} {F_{nk + nN + jnq} } }}} \right]}\,,
\end{split}
\end{equation}
from which Theorem~\ref{thm.xd29ih0} follows as $N$ approaches infinity.

\bigskip

If $q$ is an odd positive integer and either $m$ or $n$ is even, then from~\eqref{equ.w1u72rp} and Lemma~\ref{finqodd} (with $m\to 2m$) we have the summation identity
\begin{equation}
\begin{split}
&\sum_{k = 1}^N {\left[ {\frac{{( - 1)^{k - 1} }}{{\prod_{j = 0}^{m - 1} {F_{nk + njq} } \prod_{j = m + 1}^{2m} {F_{nk + njq} } }}} \right]}\\
&\qquad= \frac{1}{{L_{mnq} }}\sum_{k = 1}^q {\left[ {\frac{{( - 1)^{k - 1} }}{{\prod_{j = 0}^{2m - 1} {F_{nk + njq} } }}} \right]}
+\frac{{( - 1)^{N - 1} }}{{L_{mnq} }}\sum_{k = 1}^q {\left[ {\frac{{( - 1)^{k - 1} }}{{\prod_{j = 0}^{2m - 1} {F_{nk + nN + njq} } }}} \right]}\,,
\end{split}
\end{equation}
from which Theorem~\ref{thm.i3apd20} follows in the limit that $N$ approaches infinity.
\begin{thm}\label{thm.cgjiqk9}
If $m$, $n$ and $q$ are positive integers such that $q$ is odd, then
\[
\sum_{k = 1}^\infty  {\frac{{( - 1)^{k - 1} F_{2nk + mnq} \prod_{j = 1}^{m - 1} {L_{nk + jnq} } }}{{\prod_{j = 0}^m {F_{nk + jnq} } }}}  = \frac{1}{2}\sum_{k = 1}^q {( - 1)^{k - 1} \prod_{j = 0}^{m - 1} {\frac{{L_{nk + jnq} }}{{F_{nk + jnq} }}} }\,. 
\]
\end{thm}
In particular,
\begin{equation}
\sum_{k = 1}^\infty  {\frac{{( - 1)^{k - 1} F_{2nk + nq} }}{{F_{nk} F_{nk + nq} }}}  = \frac{1}{2}\sum_{k = 1}^q {( - 1)^{k - 1} \frac{{L_{nk} }}{{F_{nk} }}},\quad\mbox{$q$ odd}\,.
\end{equation}
\begin{cor}\label{thm.ee6ees5}
If $m$, $n$ and $q$ are positive integers such that $q$ is odd, then
\[
\sum_{k = 1}^\infty  {\frac{{( - 1)^{k - 1} L_{nk + mnq}^2 \prod_{j = 1}^{m - 1} {L_{nk + jnq} } \prod_{j = m + 1}^{2m - 1} {L_{nk + jnq} } }}{{\prod_{j = 0}^{m - 1} {F_{nk + jnq} } \prod_{j = m + 1}^{2m} {F_{nk + jnq} } }}}  = \frac{1}{2}\sum_{k = 1}^q {( - 1)^{k - 1} \prod_{j = 0}^{2m - 1} {\frac{{L_{nk + jnq} }}{{F_{nk + jnq} }}} }\,.
\]

\end{cor}
In particular,
\begin{equation}
\sum_{k = 1}^\infty  {\frac{{( - 1)^{k - 1} L_{nk + nq}^2 }}{{F_{nk} F_{nk + 2nq} }}}  = \frac{1}{2}\sum_{k = 1}^q {( - 1)^{k - 1} \frac{{L_{nk} L_{nk + nq} }}{{F_{nk} F_{nk + nq} }}},\quad\mbox{$q$ odd}\,.
\end{equation}
\subsubsection*{Proof of Theorem~\ref{thm.cgjiqk9} and Corollary~\ref{thm.ee6ees5}}
Dividing through the identity~\eqref{equ.ejrnkwy} by $F_uF_v$ and setting $u=nk+mnq$ and $v=nk$ we have the identity
\begin{equation}\label{equ.klrep5j}
2\frac{{F_{2nk + mnq} }}{{F_{nk} F_{nk + mnq} }} = \frac{{L_{nk + mnq} }}{{F_{nk + mnq} }} + \frac{{L_{nk} }}{{F_{nk} }}\,.
\end{equation}
If $q$ is odd, then from~\eqref{equ.klrep5j} and with $f(k)=L_k/F_k$ in Lemma~\ref{finqodd} we have
\begin{equation}\label{equ.bo7m7gb}
\begin{split}
&2\sum_{k = 1}^N {\frac{{( - 1)^{k - 1} F_{2nk + mnq} \prod_{j = 1}^{m - 1} {L_{nk + jnq} } }}{{\prod_{j = 0}^m {F_{nk + jnq} } }}}\\
&\qquad\qquad= \sum_{k = 1}^q {( - 1)^{k - 1} \prod_{j = 0}^{m - 1} {\frac{{L_{nk + jnq} }}{{F_{nk + jnq} }}} }  + ( - 1)^{N - 1} \sum_{k = 1}^q {( - 1)^{k - 1} \prod_{j = 0}^{m - 1} {\frac{{L_{nk + nN + jnq} }}{{F_{nk + nN + jnq} }}} }\,,
\end{split}
\end{equation}
from which Theorem~\eqref{thm.cgjiqk9} follows in the limit as $N$ approaches infinity. Corollary~\eqref{thm.ee6ees5} is obtained by specifically requiring $m$ to be even in Theorem~\eqref{thm.cgjiqk9}.
\begin{thm}\label{thm.f9a8imt}
If $m$, $n$, $q$ and $p$ are positive integers, then
\[
\begin{split}
&\sum_{k = 1}^\infty  {\left\{ {\frac{{( - 1)^{nk - 1} \prod_{j = 1}^{m - 1} {F_{nk + jnq + np} } }}{{\prod_{j = 0}^m {F_{nk + jnq} } }}} \right\}}\\
&\qquad= \frac{{\phi ^{mnp} q}}{{F_{mnq} F_{np} }} - \frac{1}{{F_{mnq} F_{np} }}\sum_{k = 1}^q {\left\{ {\prod_{j = 0}^{m - 1} {\frac{{F_{nk + jnq + np} }}{{F_{nk + jnq} }}} } \right\}}\,.
\end{split}
\]

\end{thm}
In particular we have
\begin{equation}
\sum_{k = 1}^\infty  {\left\{ {\frac{{( - 1)^{nk - 1} \prod_{j = m + 1}^{p + m - 1} {F_{nk + jn} } }}{{\prod_{j = 0}^p {F_{nk + jn} } }}} \right\}}  = \frac{{\phi ^{mnp} }}{{F_{mn} F_{pn} }} - \frac{1}{{F_{mn} F_{pn} }}\prod_{j = 0}^{m - 1} {\frac{{F_{jn + np + n} }}{{F_{jn + n} }}}
\end{equation}
and
\begin{equation}\label{equ.b2p05i3}
\sum_{k = 1}^\infty  {\frac{{( - 1)^{nk - 1} }}{{F_{nk} F_{nk + nq} }}}  = \frac{{\phi ^{n} q}}{{F_{nq} F_{n} }} - \frac{1}{{F_{nq} F_{n} }}\sum_{k = 1}^q {\frac{{F_{nk + n} }}{{F_{nk} }}}\,.
\end{equation}
Observe that identity~\eqref{equ.b2p05i3} is equivalent to identity~\eqref{equ.xk94y11} but with a different form for the right hand side. Since $2\phi=\sqrt 5+1$, \mbox{$F_{n-1}+F_{n+1}=L_n$} and $\phi^n = \phi F_n+F_{n-1}$, both identities can be combined to yield the following interesting summation identity which is valid for all non-zero integers~$n$ and non-negative integers~$q$:
\begin{equation}
\frac{1}{{F_n }}\sum\limits_{k = 1}^q {\left[ {\frac{{F_{nk + n} }}{{F_{nk} }}} \right]}  - \frac{1}{2}\sum\limits_{k = 1}^q {\left[ {\frac{{L_{nk} }}{{F_{nk} }}} \right]}  = \frac{q}{2}\frac{{L_n }}{{F_n }}\,.
\end{equation}
\begin{thm}\label{thm.vdv26oc}
If $m$, $n$, $q$ and $p$ are positive integers such that $n$ is odd and $q$ is even, then
\[
\begin{split}
&\sum_{k = 1}^\infty  {\left\{ {\frac{{ \prod_{j = 1}^{m - 1} {F_{nk + jnq + np} } }}{{\prod_{j = 0}^m {F_{nk + jnq} } }}} \right\}}\\
&\qquad=\frac{1}{{F_{mnq} F_{np} }}\sum_{k = 1}^q {\left\{ {( - 1)^{k}\prod_{j = 0}^{m - 1} {\frac{{F_{nk + jnq + np} }}{{F_{nk + jnq} }}} } \right\}}\,.
\end{split}
\]

\end{thm}
In particular,
\begin{equation}\label{equ.y9fdju9}
\sum_{k = 1}^\infty  {\frac{1}{{F_{nk} F_{nk + nq} }} = \frac{1}{{F_nF_{nq} }}\sum_{k = 1}^q {\left[ {( - 1)^k \frac{{F_{nk+n} }}{{F_{nk} }}} \right]} },\quad\mbox{$n$ odd, $q$ even}\,.
\end{equation}
From identity~\eqref{equ.lldx63d} and identity~\eqref{equ.y9fdju9} we have the interesting result
\begin{equation}
\frac{1}{2}\sum_{k = 1}^q {\left[ {( - 1)^k \frac{{L_{nk} }}{{F_{nk} }}} \right]}  = \frac{1}{{F_n }}\sum_{k = 1}^q {\left[ {( - 1)^k \frac{{F_{nk + n} }}{{F_{nk} }}} \right]}\,,\quad\mbox{$q$ even}\,.
\end{equation}
\subsection*{Proof of Theorem~\ref{thm.f9a8imt} and Theorem~\ref{thm.vdv26oc}}
Dividing through identity~\eqref{equ.tokbcvq} by $F_{t+u}F_t$ and choosing $t=nk$, $u=mnq$ and $v=np$ we obtain the identity:
\begin{equation}\label{equ.oi2ipiz}
\frac{{( - 1)^{nk - 1} F_{mnq} F_{np} }}{{F_{nk} F_{nk + mnq} }} = \frac{{F_{nk + mnq + np} }}{{F_{nk + mnq} }} - \frac{{F_{nk + np} }}{{F_{nk} }}\,.
\end{equation}
With $f(k)=F_{k+np}/F_k$ in Lemma~\ref{finall} and using the identity~\eqref{equ.oi2ipiz}, we have the finite summation identity
\begin{equation}
\begin{split}
&F_{mnq} F_{np} \sum_{k = 1}^N {\left\{ {\frac{{( - 1)^{nk - 1} \prod_{j = 1}^{m - 1} {F_{nk + jnq + np} } }}{{\prod_{j = 0}^m {F_{nk + jnq} } }}} \right\}}\\
&\quad= \sum_{k = 1}^q {\left\{ {\prod_{j = 0}^{m - 1} {\frac{{F_{nk + nN + jnq + np} }}{{F_{nk + nN + jnq} }}} } \right\}}  - \sum_{k = 1}^q {\left\{ {\prod_{j = 0}^{m - 1} {\frac{{F_{nk + jnq + np} }}{{F_{nk + jnq} }}} } \right\}}\,,
\end{split}
\end{equation}
from which Theorem~\ref{thm.f9a8imt} follows in the limit as $N$ approaches infinity. 

\bigskip

Using $f(k)=F_{k+np}/F_k$ in Lemma~\ref{finqeven} gives
\begin{equation}
\begin{split}
&F_{mnq} F_{np} \sum_{k = 1}^N {\left\{ {\frac{{\prod_{j = 1}^{m - 1} {F_{nk + jnq + np} } }}{{\prod_{j = 0}^m {F_{nk + jnq} } }}} \right\}}\\ 
&\quad = ( - 1)^N \sum_{k = 1}^q {\left\{ {( - 1)^{k - 1} \prod_{j = 0}^{m - 1} {\frac{{F_{nk + nN + jnq + np} }}{{F_{nk + nN + jnq} }}} } \right\}}\\ 
&\qquad - \sum_{k = 1}^q {\left\{ {( - 1)^{k - 1} \prod_{j = 0}^{m - 1} {\frac{{F_{nk + jnq + np} }}{{F_{nk + jnq} }}} } \right\}}\,, 
\end{split}
\end{equation}
from which Theorem~\ref{thm.vdv26oc} follows.
\subsection{Sums with $L_{nk}L_{nk+nq}\cdots L_{nk+mnq}$ or\\ $L_{nk}L_{nk+nq}\cdots L_{nk+mnq-nq}L_{nk+mnq+nq}\cdots L_{nk+2mnq}$ in the denominator}
The derivations here proceed in the same fashion as in the previous section. The theorems will therefore be stated without proof. The analogous identity to~\eqref{equ.oy4215f} is 
\begin{equation}\label{equ.rcy3566}
\begin{split}
&2F_{mnq} \sum_{k = 1}^N {\left[ {( - 1)^{nk - 1} \frac{{\prod_{j = 1}^{m - 1} {F_{nk + jnq} } }}{{\prod_{j = 0}^m {L_{nk + jnq} } }}} \right]}\\
&\quad  = \sum_{k = 1}^q {\left[ {\prod_{j = 0}^{m - 1} {\frac{{F_{nk + jnq} }}{{L_{nk + jnq} }}} } \right]}-\sum_{k = 1}^q {\left[ {\prod_{j = 0}^{m - 1} {\frac{{F_{nk + nN + jnq} }}{{L_{nk + nN + jnq} }}} } \right]}\,.
\end{split}
\end{equation}
\begin{thm}\label{thm.v8d2lop}
If $m$, $n$ and $q$ are positive integers, then
\[
\sum_{k = 1}^\infty  {\left[ {( - 1)^{nk - 1} \frac{{\prod_{j = 1}^{m - 1} {F_{nk + jnq} } }}{{\prod_{j = 0}^m {L_{nk + jnq} } }}} \right]}  = \frac{1}{{2F_{mnq} }}\sum_{k = 1}^q {\prod_{j = 0}^{m - 1} {\frac{{F_{nk + jnq} }}{{L_{nk + jnq} }}} }-\frac{{q }}{{2F_{mnq}\sqrt{ 5^m} }}\,, 
\]
so that
\begin{equation}
\sum_{k = 1}^\infty  {\left[ {\frac{{\prod_{j = 1}^{m - 1} {F_{nk + jnq} } }}{{\prod_{j = 0}^m {L_{nk + jnq} } }}} \right]}  = \frac{{q }}{{2F_{mnq}\sqrt{ 5^m} }}-\frac{1}{{2F_{mnq} }}\sum_{k = 1}^q {\prod_{j = 0}^{m - 1} {\frac{{F_{nk + jnq} }}{{L_{nk + jnq} }}} },\quad\mbox{$n$ even}
\end{equation}
and
\begin{equation}
\sum_{k = 1}^\infty  {\left[ {( - 1)^{k - 1} \frac{{\prod_{j = 1}^{m - 1} {F_{nk + jnq} } }}{{\prod_{j = 0}^m {L_{nk + jnq} } }}} \right]}  = \frac{1}{{2F_{mnq} }}\sum_{k = 1}^q {\prod_{j = 0}^{m - 1} {\frac{{F_{nk + jnq} }}{{L_{nk + jnq} }}} }-\frac{{q }}{{2F_{mnq}\sqrt{ 5^m} }} ,\quad\mbox{$n$ odd}\,.
\end{equation}

\end{thm}
In particular,
\begin{equation}\label{equ.elnysff}
\sum_{k = 1}^\infty  {\frac{{( - 1)^{nk - 1} }}{{L_{nk} L_{nk + nq} }}}  =\frac{1}{{2F_{nq} }}\sum_{k = 1}^q {\frac{{F_{nk} }}{{L_{nk} }}}-\frac{{q }}{{2F_{nq}\sqrt{ 5} }},\quad n,q\in\Zplus\,.
\end{equation}
The case $n=3,q=1$ in~\eqref{equ.elnysff} was mentioned by Brousseau (equation~(14) of~\cite{brousseau1}).
\begin{thm}\label{thm.llpbu3h}
If $m$, $n$ and $q$ are integers such that $n$ is odd and $q$ is even, then
\[
\sum_{k = 1}^\infty  {\left[ {\frac{{\prod_{j = 1}^{m - 1} {F_{nk + jnq} } }}{{\prod_{j = 0}^m {L_{nk + jnq} } }}} \right]}  = \frac{1}{{2F_{mnq} }}\sum_{k = 1}^q {\left[ {( - 1)^{k-1} \prod_{j = 0}^{m - 1} {\frac{{F_{nk + jnq} }}{{L_{nk + jnq} }}} } \right]}\,. 
\]
\end{thm}
In particular,
\begin{equation}\label{equ.x6dcfpg}
\sum_{k = 1}^\infty  {\frac{1}{{L_{nk} L_{nk + nq} }} = \frac{1}{{2F_{nq} }}\sum_{k = 1}^q {\left[ {( - 1)^{k-1} \frac{{F_{nk} }}{{L_{nk} }}} \right]} },\quad\mbox{$n$ odd, $q$ even}\,.
\end{equation}

\begin{thm}\label{thm.jysuj3n}
If $m$, $n$ and $q$ are positive \underline{odd} integers, then
\[
\sum_{k = 1}^\infty  {\left[ {\frac{1}{{\prod_{j = 0}^{m - 1} {L_{nk + njq} } \prod_{j = m + 1}^{2m} {L_{nk + njq} } }}} \right]}  = \frac{1}{{L_{mnq} }}\sum_{k = 1}^q {\left[ {\frac{1}{{\prod_{j = 0}^{2m - 1} {L_{nk + njq} } }}} \right]}\,. 
\]

\end{thm}

\begin{thm}\label{thm.eniunx3}
If $q$, $m$ and $n$ are positive integers such that $q$ is odd and $nm$ is even, then
\[
\sum_{k = 1}^\infty  {\left[ {\frac{(-1)^{k-1}}{{\prod_{j = 0}^{m - 1} {L_{nk + njq} } \prod_{j = m + 1}^{2m} {L_{nk + njq} } }}} \right]}  = \frac{1}{{L_{mnq} }}\sum_{k = 1}^q {\left[ {\frac{(-1)^{k-1}}{{\prod_{j = 0}^{2m - 1} {L_{nk + njq} } }}} \right]}\,. 
\]
\end{thm}

Analogous identity to identity~\eqref{equ.bo7m7gb} is
\begin{equation}\label{equ.s11wekj}
\begin{split}
&2\sum_{k = 1}^N {\frac{{( - 1)^{k - 1} F_{2nk + mnq} \prod_{j = 1}^{m - 1} {F_{nk + jnq} } }}{{\prod_{j = 0}^m {L_{nk + jnq} } }}}\\
&\qquad\qquad= \sum_{k = 1}^q {( - 1)^{k - 1} \prod_{j = 0}^{m - 1} {\frac{{F_{nk + jnq} }}{{L_{nk + jnq} }}} }  + ( - 1)^{N - 1} \sum_{k = 1}^q {( - 1)^{k - 1} \prod_{j = 0}^{m - 1} {\frac{{F_{nk + nN + jnq} }}{{L_{nk + nN + jnq} }}} }\,,
\end{split}
\end{equation}
from which we get the following theorem in the limit as $N$ approaches infinity.
\begin{thm}\label{thm.oe85fcz}
If $m$, $n$ and $q$ are positive integers such that $q$ is odd, then
\[
\sum_{k = 1}^\infty  {\frac{{( - 1)^{k - 1} F_{2nk + mnq} \prod_{j = 1}^{m - 1} {F_{nk + jnq} } }}{{\prod_{j = 0}^m {L_{nk + jnq} } }}}  = \frac{1}{2}\sum_{k = 1}^q {( - 1)^{k - 1} \prod_{j = 0}^{m - 1} {\frac{{F_{nk + jnq} }}{{L_{nk + jnq} }}} }\,. 
\]
\end{thm}
\begin{cor}\label{thm.q95bigi}
If $m$, $n$ and $q$ are positive integers such that $q$ is odd, then
\[
\sum_{k = 1}^\infty  {\frac{{( - 1)^{k - 1} F_{nk + mnq}^2 \prod_{j = 1}^{m - 1} {F_{nk + jnq} } \prod_{j = m + 1}^{2m - 1} {F_{nk + jnq} } }}{{\prod_{j = 0}^{m - 1} {L_{nk + jnq} } \prod_{j = m + 1}^{2m} {L_{nk + jnq} } }}}  = \frac{1}{2}\sum_{k = 1}^q {( - 1)^{k - 1} \prod_{j = 0}^{2m - 1} {\frac{{F_{nk + jnq} }}{{L_{nk + jnq} }}} }\,.
\]

\end{cor}
Corresponding to identity~\eqref{equ.oi2ipiz} of section~\ref{sec.u8lgixm} we have (from identity~\eqref{equ.owip3z8})
\begin{equation}
\frac{{( - 1)^{nk - 1} 5F_{mnq} F_{np} }}{{L_{nk} L_{nk + mnq} }} = -\frac{{L_{nk + mnq + np} }}{{L_{nk + mnq} }} + \frac{{L_{nk + np} }}{{L_{nk} }}\,,
\end{equation}
leading to the summation identities
\begin{equation}
\begin{split}
&5F_{mnq} F_{np} \sum_{k = 1}^N {\left\{ {\frac{{( - 1)^{nk - 1} \prod_{j = 1}^{m - 1} {L_{nk + jnq + np} } }}{{\prod_{j = 0}^m {L_{nk + jnq} } }}} \right\}}\\
&\quad= -\sum_{k = 1}^q {\left\{ {\prod_{j = 0}^{m - 1} {\frac{{L_{nk + nN + jnq + np} }}{{L_{nk + nN + jnq} }}} } \right\}}  + \sum_{k = 1}^q {\left\{ {\prod_{j = 0}^{m - 1} {\frac{{L_{nk + jnq + np} }}{{L_{nk + jnq} }}} } \right\}}
\end{split}
\end{equation}
and
\begin{equation}
\begin{split}
&5F_{mnq} F_{np} \sum_{k = 1}^N {\left\{ {\frac{{\prod_{j = 1}^{m - 1} {L_{nk + jnq + np} } }}{{\prod_{j = 0}^m {L_{nk + jnq} } }}} \right\}}\\ 
&\quad = ( - 1)^{N-1} \sum_{k = 1}^q {\left\{ {( - 1)^{k - 1} \prod_{j = 0}^{m - 1} {\frac{{L_{nk + nN + jnq + np} }}{{L_{nk + nN + jnq} }}} } \right\}}\\ 
&\qquad + \sum_{k = 1}^q {\left\{ {( - 1)^{k - 1} \prod_{j = 0}^{m - 1} {\frac{{L_{nk + jnq + np} }}{{L_{nk + jnq} }}} } \right\}}\,, 
\end{split}
\end{equation}
from which Theorem~\ref{thm.lk7ts7r} and Theorem~\ref{thm.ec1t6t2} follow.
\begin{thm}\label{thm.lk7ts7r}
If $m$, $n$, $q$ and $p$ are positive integers, then
\[
\begin{split}
&\sum_{k = 1}^\infty  {\left\{ {\frac{{( - 1)^{nk - 1} \prod_{j = 1}^{m - 1} {L_{nk + jnq + np} } }}{{\prod_{j = 0}^m {L_{nk + jnq} } }}} \right\}}\\
&\qquad= -\frac{{\phi ^{mnp} q}}{{5F_{mnq} F_{np} }} + \frac{1}{{5F_{mnq} F_{np} }}\sum_{k = 1}^q {\left\{ {\prod_{j = 0}^{m - 1} {\frac{{L_{nk + jnq + np} }}{{L_{nk + jnq} }}} } \right\}}\,.
\end{split}
\]

\end{thm}
In particular we have
\begin{equation}
\sum_{k = 1}^\infty  {\left\{ {\frac{{( - 1)^{nk - 1} \prod_{j = m + 1}^{p + m - 1} {L_{nk + jn} } }}{{\prod_{j = 0}^p {L_{nk + jn} } }}} \right\}}  = -\frac{{\phi ^{mnp} }}{{5F_{mn} F_{pn} }} + \frac{1}{{5F_{mn} F_{pn} }}\prod_{j = 0}^{m - 1} {\frac{{L_{jn + np + n} }}{{L_{jn + n} }}}
\end{equation}
and
\begin{equation}
\sum_{k = 1}^\infty  {\frac{{( - 1)^{nk - 1} }}{{L_{nk} L_{nk + nq} }}}  = -\frac{{\phi ^{n} q}}{{5F_{nq} F_{n} }} + \frac{1}{{5F_{nq} F_{n} }}\sum_{k = 1}^q {\frac{{L_{nk + n} }}{{L_{nk} }}}\,.
\end{equation}

\begin{thm}\label{thm.ec1t6t2}
If $m$, $n$, $q$ and $p$ are positive integers such that $n$ is odd and $q$ is even, then
\[
\begin{split}
&\sum_{k = 1}^\infty  {\left\{ {\frac{{ \prod_{j = 1}^{m - 1} {L_{nk + jnq + np} } }}{{\prod_{j = 0}^m {L_{nk + jnq} } }}} \right\}}\\
&\qquad=\frac{1}{{5F_{mnq} F_{np} }}\sum_{k = 1}^q {\left\{ {( - 1)^{k-1}\prod_{j = 0}^{m - 1} {\frac{{L_{nk + jnq + np} }}{{L_{nk + jnq} }}} } \right\}}\,.
\end{split}
\]

\end{thm}
In particular,
\begin{equation}\label{equ.a0f8aex}
\sum_{k = 1}^\infty  {\frac{1}{{L_{nk} L_{nk + nq} }} = \frac{1}{{5F_nF_{nq} }}\sum_{k = 1}^q {\left[ {( - 1)^{k-1} \frac{{L_{nk+n} }}{{L_{nk} }}} \right]} },\quad\mbox{$n$ odd, $q$ even}\,.
\end{equation}
From identity~\eqref{equ.x6dcfpg} and identity~\eqref{equ.a0f8aex} we have
\begin{equation}
\frac{1}{2}\sum_{k = 1}^q {\left[ {( - 1)^{k-1} \frac{{F_{nk} }}{{L_{nk} }}} \right]}  = \frac{1}{{5F_n }}\sum_{k = 1}^q {\left[ {( - 1)^{k-1} \frac{{L_{nk + n} }}{{L_{nk} }}} \right]}\,,\quad\mbox{$q$ even}\,.
\end{equation}

\subsection{Sums with $F_{nk}F_{nk+nq}\cdots F_{nk+2mnq}$ in the denominator}
The results in this section are obtained from identity~\eqref{equ.yb05ue2}. We have
\[
\frac{{F_v L_u }}{{F_{u - v} F_{u + v} }}=\frac{1}{{F_{u - v} }} - \frac{{( - 1)^v }}{{F_{u + v} }}\,,
\]
from which, by setting $v=mnq$ and $u=nk+mnq$, we get
\begin{equation}
\frac{{F_{mnq} L_{nk + mnq} }}{{F_{nk} F_{nk+2mnq} }} = \frac{1}{{F_{nk} }} - \frac{{( - 1)^{mnq} }}{{F_{nk+2mnq} }}\,,
\end{equation}
so that
\begin{equation}\label{equ.srgqhu2}
\frac{{F_{mnq} L_{nk + mnq} }}{{F_{nk} F_{nk+2mnq} }} = \frac{1}{{F_{nk} }} - \frac{1}{{F_{nk+2mnq} }},\quad\quad\mbox{$mnq$ even}
\end{equation}
and
\begin{equation}\label{equ.ih5bg09}
\frac{{F_{mnq} L_{nk + mnq} }}{{F_{nk} F_{nk+2mnq} }} = \frac{1}{{F_{nk} }} + \frac{1}{{F_{nk+2mnq} }},\quad\mbox{$mnq$ odd}\,.
\end{equation}
The derivations then proceed as in the previous sections.
\begin{thm}\label{thm.k2v2cn5}
If $m$, $n$ and $q$ are positive integers such that $mnq$ is even, then
\[
\sum_{k = 1}^\infty  {\left[ {\frac{{L_{nk + mnq} }}{{\prod_{j = 0}^{2m} {F_{nk + njq} } }}} \right]}  = \frac{1}{{F_{mnq} }}\sum_{k = 1}^q {\left[ {\frac{1}{{\prod_{j = 0}^{2m - 1} {F_{nk + njq} } }}} \right]} 
\]

\end{thm}
Examples from Theorem~\ref{thm.k2v2cn5} include:
\begin{equation}\label{equ.qchc8p4}
\begin{split}
&\mbox{At $(m,n,q)=(1,2,1)$ and $(m,n,q)=(1,1,2)$:}\\
&\sum_{k = 1}^\infty  {\frac{{L_{2k + 2} }}{{F_{2k} F_{2k + 2} F_{2k + 4} }}}  = \frac{1}{3},\quad\sum_{k = 1}^\infty  {\frac{{L_{k + 2} }}{{F_k F_{k + 2} F_{k + 4} }}}  = \frac{5}{6}\,.
\end{split}
\end{equation}
The first of the identities in~\eqref{equ.qchc8p4} was also derived in~\cite{melham01} (equation~3.7).
\[
\begin{split}
&\mbox{At $(m,n,q)=(2,7,1)$:}\\
&\sum_{k = 1}^\infty  {\frac{{L_{7k + 14} }}{{F_{7k} F_{7k + 7} F_{7k + 14} F_{7k + 21} F_{7k + 28} }}}  = \frac{1}{{{\rm 6427623373464462}}}\,.
\end{split}
\]
\begin{thm}\label{thm.n3errmo}
If $m$, $n$ and $q$ are positive integers such that $q$ is even or $mnq$ is odd, then
\[
\sum_{k = 1}^\infty  {\left[ {\frac{{(-1)^{k-1}L_{nk + mnq} }}{{\prod_{j = 0}^{2m} {F_{nk + njq} } }}} \right]}  = \frac{1}{{F_{mnq} }}\sum_{k = 1}^q {\left[ {\frac{(-1)^{k-1}}{{\prod_{j = 0}^{2m - 1} {F_{nk + njq} } }}} \right]} 
\]

\end{thm}
Examples from Theorem~\ref{thm.n3errmo} include:
\[
\begin{split}
&\mbox{At $(m,n,q)=(1,1,2)$ and $(m,n,q)=(1,3,2)$:}\\
&\sum_{k = 1}^\infty  {\frac{{( - 1)^{k - 1} L_{k + 2} }}{{F_k F_{k + 2} F_{k + 4} }}}  = \frac{1}{6}\,,\quad \sum_{k = 1}^\infty  {\frac{{( - 1)^{k - 1} L_{3k + 6} }}{{F_{3k} F_{3k + 6} F_{3k + 12} }}}  = \frac{{271}}{{156672}}\,.
\end{split}
\]
\[
\begin{split}
&\mbox{At $(m,n,q)=(2,4,2)$:}\\
&\sum_{k = 1}^\infty  {\frac{{( - 1)^{k - 1} L_{4k + 16} }}{{F_{4k} F_{4k + 8} F_{4k + 16} F_{4k + 24} F_{4k + 32} }}}  = \frac{{{\rm 177072540680427}}}{{{\rm 166704475185956548320480}}}\,.
\end{split}
\]

\subsection{Sums with $L_{nk}L_{nk+nq}\cdots L_{nk+2mnq}$ in the denominator}
The results here follow from the identity~\eqref{equ.q79u3q4}.
\begin{thm}\label{thm.h6ute1y}
If $m$, $n$ and $q$ are positive integers such that $mnq$ is even, then
\[
\sum_{k = 1}^\infty  {\left[ {\frac{{F_{nk + mnq} }}{{\prod_{j = 0}^{2m} {L_{nk + jnq} } }}} \right]}  =\frac1{5F_{mnq}} \sum_{k = 1}^q {\left[ {\prod_{j = 0}^{2m - 1} {\frac{1}{{L_{nk + jnq} }}} } \right]}\,. 
\]
\end{thm}

\begin{thm}\label{thm.gkxbzjk}
If $m$, $n$ and $q$ are positive integers such that $q$ is even or $mnq$ is odd, then
\[
\sum_{k = 1}^\infty  {(-1)^{k-1}\left[ {\frac{{F_{nk + mnq} }}{{\prod_{j = 0}^{2m} {L_{nk + jnq} } }}} \right]}  = \frac1{5F_{mnq}}\sum_{k = 1}^q {\left[(-1)^{k-1} {\prod_{j = 0}^{2m - 1} {\frac{1}{{L_{nk + jnq} }}} } \right]}\,. 
\]
\end{thm}

\subsection{Sums with $F_{nk}F_{nk+2nq}F_{nk+4nq}F_{nk+6nq}\cdots F_{nk+2mnq}$ in the denominator}\label{sec.pgpoxnq}
\begin{thm}\label{thm.t7s7spm}
If $m$, $n$ and $q$ are positive \underline{odd} integers, then
\[
\sum_{k=1}^\infty \frac{(\pm1)^{k-1}F_{nk+mnq}}{\prod_{j=0}^m F_{nk+2jnq}}=\frac1{L_{mnq}}\sum_{k=1}^{2q}\frac{(\pm1)^{k-1}}{\prod_{j=0}^{m-1}F_{nk+2jnq}}\,.
\]
\end{thm}
In particular,
\begin{equation}
\sum_{k = 1}^\infty  {\frac{{F_{nk + nq} }}{{F_{nk} F_{nk + 2nq} }}}  = \frac{1}{{L_{nq} }}\sum_{k = 1}^{2q} {\frac{1}{{F_{nk} }}},\quad \mbox{$nq$ odd}\,.
\end{equation}
\begin{proof}
From identity~\eqref{equ.eo7cizi} and with $f(k)=1/F_k$ (and $q\to 2q$) in Lemma~\ref{finall} we have the finite summation identity
\begin{equation}\label{equ.u640lbl}
\begin{split}
L_{mnq} \sum_{k = 1}^N {\frac{{F_{nk + mnq} }}{{\prod_{j = 0}^m {F_{nk + 2jnq} } }}}  &= \sum_{k = 1}^{2q} {\frac{1}{{\prod_{j = 0}^{m - 1} {F_{nk + 2jnq} } }}}\qquad\qquad\qquad\mbox{($mnq$ odd)}\\
&\qquad  - \sum_{k = 1}^{2q} {\frac{1}{{\prod_{j = 0}^{m - 1} {F_{nk + nN + 2jnq} } }}}\,.
\end{split}
\end{equation}
From identity~\eqref{equ.eo7cizi} with $m$, $n$ and $q$ positive odd intergers and with $f(k)=1/F_k$ (and $q\to 2q$) in Lemma~\ref{finqeven} we have the alternating finite summation identity
\begin{equation}\label{equ.bdnv59h}
\begin{split}
L_{mnq} \sum_{k = 1}^N {\frac{{( - 1)^{k - 1} F_{nk + mnq} }}{{\prod_{j = 0}^m {F_{nk + 2jnq} } }}}  &= \sum_{k = 1}^{2q} {\frac{{( - 1)^{k - 1} }}{{\prod_{j = 0}^{m - 1} {F_{nk + 2jnq} } }}}\qquad\qquad\qquad\mbox{($mnq$ odd)}\\
&\qquad  + ( - 1)^{N - 1} \sum_{k = 1}^{2q} {\frac{{( - 1)^{k - 1} }}{{\prod_{j = 0}^{m - 1} {F_{nk + nN + 2jnq} } }}}
\end{split}
\end{equation}
Theorem~\ref{thm.t7s7spm} follows from identities~\eqref{equ.u640lbl} and~\eqref{equ.bdnv59h} in the limit as $N$ approaches infinity.
\end{proof}

\begin{thm}
If $m$, $n$ and $q$ are positive integers such that $mnq$ is even, then
\[
\sum_{k=1}^\infty \frac{(\pm1)^{k-1}L_{nk+mnq}}{\prod_{j=0}^m F_{nk+2jnq}}=\frac1{F_{mnq}}\sum_{k=1}^{2q}\frac{(\pm1)^{k-1}}{\prod_{j=0}^{m-1}F_{nk+2jnq}}\,.
\]
\end{thm}
In particular,
\begin{equation}
\sum_{k = 1}^\infty  {\frac{{L_{nk + nq} }}{{F_{nk} F_{nk + 2nq} }}}  = \frac{1}{{F_{nq} }}\sum_{k = 1}^{2q} {\frac{1}{{F_{nk} }}},\quad \mbox{$nq$ even}\,.
\end{equation}
\begin{proof}
As in Theorem~\ref{thm.t7s7spm}, with the identity~\eqref{equ.srgqhu2}, with $mnq$ even.

\bigskip

The corresponding finite summation identities are
\begin{equation}\label{equ.medhcfp}
\begin{split}
F_{mnq} \sum_{k = 1}^N {\frac{{L_{nk + mnq} }}{{\prod_{j = 0}^m {F_{nk + 2jnq} } }}}  &= \sum_{k = 1}^{2q} {\frac{1}{{\prod_{j = 0}^{m - 1} {F_{nk + 2jnq} } }}}\qquad\qquad\qquad\mbox{($mnq$ even)}\\
&\qquad  - \sum_{k = 1}^{2q} {\frac{1}{{\prod_{j = 0}^{m - 1} {F_{nk + nN + 2jnq} } }}}
\end{split}
\end{equation}
and
\begin{equation}\label{equ.pmefb69}
\begin{split}
F_{mnq} \sum_{k = 1}^N {\frac{{( - 1)^{k - 1} L_{nk + mnq} }}{{\prod_{j = 0}^m {F_{nk + 2jnq} } }}}  &= \sum_{k = 1}^{2q} {\frac{{( - 1)^{k - 1} }}{{\prod_{j = 0}^{m - 1} {F_{nk + 2jnq} } }}}\qquad\qquad\qquad\mbox{($mnq$ even)}\\
&\qquad  + ( - 1)^{N - 1} \sum_{k = 1}^{2q} {\frac{{( - 1)^{k - 1} }}{{\prod_{j = 0}^{m - 1} {F_{nk + nN + 2jnq} } }}}\,.
\end{split}
\end{equation}
\end{proof}

\subsection{Sums with $L_{nk}L_{nk+2nq}L_{nk+4nq}L_{nk+6nq}\cdots L_{nk+2mnq}$ in the denominator}
In this section we state the Lucas versions of the results given in section~\ref{sec.pgpoxnq}.

\bigskip

Here the basic identities (from identities~\eqref{equ.epvyp3u}) are:
\begin{equation}\label{equ.rgq8asw}
\frac{{L_{mnq} L_{nk + mnq} }}{{L_{nk} L_{nk + 2mnq} }} = \frac{1}{{L_{nk} }} - \frac{1}{{L_{nk + 2mnq} }},\quad\mbox{$mnq$ odd}
\end{equation}
\begin{equation}\label{equ.p4ho3jr}
\frac{{5F_{mnq} F_{nk + mnq} }}{{L_{nk} L_{nk + 2mnq} }} = \frac{1}{{L_{nk} }} - \frac{1}{{L_{nk + 2mnq} }},\quad\mbox{$mnq$ even}\,.
\end{equation}
\begin{thm}
If $m$, $n$ and $q$ are positive odd integers, then
\[
\sum_{k=1}^\infty \frac{(\pm1)^{k-1}L_{nk+mnq}}{\prod_{j=0}^m L_{nk+2jnq}}=\frac1{L_{mnq}}\sum_{k=1}^{2q}\frac{(\pm1)^{k-1}}{\prod_{j=0}^{m-1}L_{nk+2jnq}}\,.
\]
\end{thm}
In particular,
\begin{equation}
\sum_{k = 1}^\infty  {\frac{{L_{nk + nq} }}{{L_{nk} L_{nk + 2nq} }}}  = \frac{1}{{L_{nq} }}\sum_{k = 1}^{2q} {\frac{1}{{L_{nk} }}},\quad \mbox{$nq$ odd}\,.
\end{equation}

\begin{thm}
If $m$, $n$ and $q$ are positive integers such that $mnq$ is even, then
\[
\sum_{k=1}^\infty \frac{(\pm1)^{k-1}F_{nk+mnq}}{\prod_{j=0}^m L_{nk+2jnq}}=\frac1{5F_{mnq}}\sum_{k=1}^{2q}\frac{(\pm1)^{k-1}}{\prod_{j=0}^{m-1}L_{nk+2jnq}}\,.
\]
\end{thm}
In particular,
\begin{equation}
\sum_{k = 1}^\infty  {\frac{{F_{nk + nq} }}{{L_{nk} L_{nk + 2nq} }}}  = \frac{1}{{5F_{nq} }}\sum_{k = 1}^{2q} {\frac{1}{{F_{nk} }}},\quad \mbox{$nq$ even}\,.
\end{equation}
We have the following finite summation identities:
\begin{equation}\label{equ.r78m8j9}
\begin{split}
L_{mnq} \sum_{k = 1}^N {\frac{{L_{nk + mnq} }}{{\prod_{j = 0}^m {L_{nk + 2jnq} } }}}  &= \sum_{k = 1}^{2q} {\frac{1}{{\prod_{j = 0}^{m - 1} {L_{nk + 2jnq} } }}}\qquad\qquad\qquad\mbox{($mnq$ odd)}\\
&\qquad  - \sum_{k = 1}^{2q} {\frac{1}{{\prod_{j = 0}^{m - 1} {L_{nk + nN + 2jnq} } }}}\,.
\end{split}
\end{equation}
\begin{equation}\label{equ.ans99dd}
\begin{split}
L_{mnq} \sum_{k = 1}^N {\frac{{( - 1)^{k - 1} L_{nk + mnq} }}{{\prod_{j = 0}^m {L_{nk + 2jnq} } }}}  &= \sum_{k = 1}^{2q} {\frac{{( - 1)^{k - 1} }}{{\prod_{j = 0}^{m - 1} {L_{nk + 2jnq} } }}}\qquad\qquad\qquad\mbox{($mnq$ odd)}\\
&\qquad  + ( - 1)^{N - 1} \sum_{k = 1}^{2q} {\frac{{( - 1)^{k - 1} }}{{\prod_{j = 0}^{m - 1} {L_{nk + nN + 2jnq} } }}}\,.
\end{split}
\end{equation}
\begin{equation}\label{equ.en25r4r}
\begin{split}
5F_{mnq} \sum_{k = 1}^N {\frac{{F_{nk + mnq} }}{{\prod_{j = 0}^m {L_{nk + 2jnq} } }}}  &= \sum_{k = 1}^{2q} {\frac{1}{{\prod_{j = 0}^{m - 1} {L_{nk + 2jnq} } }}}\qquad\qquad\qquad\mbox{($mnq$ even)}\\
&\qquad  - \sum_{k = 1}^{2q} {\frac{1}{{\prod_{j = 0}^{m - 1} {L_{nk + nN + 2jnq} } }}}\,.
\end{split}
\end{equation}
\begin{equation}\label{equ.qa1qok6}
\begin{split}
5F_{mnq} \sum_{k = 1}^N {\frac{{( - 1)^{k - 1} F_{nk + mnq} }}{{\prod_{j = 0}^m {L_{nk + 2jnq} } }}}  &= \sum_{k = 1}^{2q} {\frac{{( - 1)^{k - 1} }}{{\prod_{j = 0}^{m - 1} {L_{nk + 2jnq} } }}}\qquad\qquad\qquad\mbox{($mnq$ even)}\\
&\qquad  + ( - 1)^{N - 1} \sum_{k = 1}^{2q} {\frac{{( - 1)^{k - 1} }}{{\prod_{j = 0}^{m - 1} {L_{nk + nN + 2jnq} } }}}\,.
\end{split}
\end{equation}

\subsection{Sums with $F_{2nk}F_{2nk+2nq}F_{2nk+4nq}F_{2nk+6nq}\cdots F_{2nk+2mnq}$ in the denominator}\label{sec.ffn2hfi}
\begin{thm}\label{thm.yj680l9}
If $m$, $n$ and $q$ are positive \underline{odd} integers, then
\[
\sum_{k = 1}^\infty  {\frac{{F_{2nk + mnq} }}{{\prod_{j = 0}^m {F_{2nk + 2jnq} } }}}  = \frac{1}{{L_{mnq} }}\sum_{k = 1}^q {\frac{1}{{\prod_{j = 0}^{m - 1} {F_{2nk + 2jnq} } }}}\,. 
\]

\end{thm}
In particular,
\begin{equation}\label{equ.akqtcqn}
\sum_{k = 1}^\infty  {\frac{{F_{2nk + nq} }}{{F_{2nk} F_{2nk + 2nq} }}}  = \frac{1}{{L_{nq} }}\sum_{k = 1}^q {\frac{1}{{F_{2nk} }}}\,,
\end{equation}
\begin{equation}
\sum_{k = 1}^\infty  {\frac{{F_{2nk + 3nq} }}{{F_{2nk} F_{2nk + 2nq} F_{2nk + 4nq} F_{2nk + 6nq} }}}  = \frac{1}{{L_{3nq} }}\sum_{k = 1}^q {\frac{1}{{F_{2nk} F_{2nk + 2nq} F_{2nk + 4nq} }}}\,.
\end{equation}
Identity~(18) of reference~\cite{frontczak16} with $n=0$ in their notation corresponds to setting $q=1$ in identity~\eqref{equ.akqtcqn}.
\begin{proof}
From identity~\eqref{equ.mzt8c69} with $v=mnq$ and $u=2nk+mnq$ comes the identity
\begin{equation}\label{equ.cx6zf5s}
\frac{{L_{mnq} F_{2nk + mnq} }}{{F_{2nk} F_{2nk + 2mnq} }} = \frac{1}{{F_{2nk} }} - \frac{1}{{F_{2nk + 2mnq} }},\quad\mbox{$mnq$ odd}\,.
\end{equation}
From identity~\eqref{equ.cx6zf5s} and Lemma~\ref{finall} with $f(k)=1/F_{2k}$ we have the finite summation identity:
\begin{equation}
\begin{split}
L_{mnq} \sum_{k = 1}^N {\frac{{F_{2nk + mnq} }}{{\prod_{j = 0}^m {F_{2nk + 2jnq} } }}}  &= \sum_{k = 1}^q {\frac{1}{{\prod_{j = 0}^{m - 1} {F_{2nk + 2jnq} } }}}\\
&\qquad  - \sum_{k = 1}^q {\frac{1}{{\prod_{j = 0}^{m - 1} {F_{2nk + 2nN + 2jnq} } }}}\,,
\end{split}
\end{equation}
from which Theorem~\ref{thm.yj680l9} follows in the limit as $N$ approaches infinity.
\end{proof}
\begin{thm}\label{thm.jarxg8t}
If $m$, $n$ and $q$ are positive integers such that $q$ is odd and $mn$ is even , then
\[
\sum_{k = 1}^\infty  {\frac{(-1)^{k-1}F_{2nk + mnq} }{{\prod_{j = 0}^m {F_{2nk + 2jnq} } }}}  = \frac{1}{{L_{mnq} }}\sum_{k = 1}^q {\frac{(-1)^{k-1}}{{\prod_{j = 0}^{m - 1} {F_{2nk + 2jnq} } }}}\,. 
\]
\end{thm}
In particular,
\begin{equation}\label{equ.dn22fd6}
\sum_{k = 1}^\infty  {\frac{{(-1)^{k-1}F_{2nk + nq} }}{{F_{2nk} F_{2nk + 2nq} }}}  = \frac{1}{{L_{nq} }}\sum_{k = 1}^q {\frac{(-1)^{k-1}}{{F_{2nk} }}}\,.
\end{equation}
\begin{proof}
From identity~\eqref{equ.mzt8c69} with $v=mnq$ and $u=2nk+mnq$ comes the identity
\begin{equation}\label{equ.dfe2pci}
\frac{{L_{mnq} F_{2nk + mnq} }}{{F_{2nk} F_{2nk + 2mnq} }} = \frac{1}{{F_{2nk} }} + \frac{1}{{F_{2nk + 2mnq} }},\quad\mbox{$mnq$ even}\,.
\end{equation}
From identity~\eqref{equ.dfe2pci} and Lemma~\ref{finqodd} with $f(k)=1/F_{2k}$ we have the finite summation identity:
\begin{equation}
\begin{split}
L_{mnq} \sum_{k = 1}^N {\frac{{(-1)^{k-1}F_{2nk + mnq} }}{{\prod_{j = 0}^m {F_{2nk + 2jnq} } }}}  &= \sum_{k = 1}^q {\frac{(-1)^{k-1}}{{\prod_{j = 0}^{m - 1} {F_{2nk + 2jnq} } }}}\\
&\qquad  +(-1)^{N-1} \sum_{k = 1}^q {\frac{(-1)^{k-1}}{{\prod_{j = 0}^{m - 1} {F_{2nk + 2nN + 2jnq} } }}}\,,
\end{split}
\end{equation}
from which Theorem~\ref{thm.jarxg8t} follows in the limit as $N$ approaches infinity.
\end{proof}
\begin{thm}\label{thm.un8i8oz}
If $m$, $n$ and $q$ are positive integers such that $mnq$ is even, then
\[
\sum_{k = 1}^\infty  {\frac{{L_{2nk + mnq} }}{{\prod_{j = 0}^m {F_{2nk + 2jnq} } }}}  = \frac{1}{{F_{mnq} }}\sum_{k = 1}^q {\frac{1}{{\prod_{j = 0}^{m - 1} {F_{2nk + 2jnq} } }}}\,. 
\]

\end{thm}
In particular,
\begin{equation}\label{equ.ci55vvi}
\sum_{k = 1}^\infty  {\frac{{L_{2nk + nq} }}{{F_{2nk} F_{2nk + 2nq} }}}  = \frac{1}{{F_{nq} }}\sum_{k = 1}^q {\frac{1}{{F_{2nk} }}}\,,
\end{equation}
\begin{equation}
\sum_{k = 1}^\infty  {\frac{{L_{2nk + 3nq} }}{{F_{2nk} F_{2nk + 2nq} F_{2nk + 4nq} F_{2nk + 6nq} }}}  = \frac{1}{{F_{3nq} }}\sum_{k = 1}^q {\frac{1}{{F_{2nk} F_{2nk + 2nq} F_{2nk + 4nq} }}}\,.
\end{equation}
\begin{proof}
From identity~\eqref{equ.yb05ue2} with $v=mnq$ and $u=2nk+mnq$ comes the identity
\begin{equation}\label{equ.cagwzby}
\frac{{F_{mnq} L_{2nk + mnq} }}{{F_{2nk} F_{2nk + 2mnq} }} = \frac{1}{{F_{2nk} }} - \frac{1}{{F_{2nk + 2mnq} }},\quad\mbox{$mnq$ even}\,.
\end{equation}
From identity~\eqref{equ.cagwzby} and Lemma~\ref{finall} with $f(k)=1/F_{2k}$ we have the finite summation identity:
\begin{equation}
\begin{split}
F_{mnq} \sum_{k = 1}^N {\frac{{L_{2nk + mnq} }}{{\prod_{j = 0}^m {F_{2nk + 2jnq} } }}}  &= \sum_{k = 1}^q {\frac{1}{{\prod_{j = 0}^{m - 1} {F_{2nk + 2jnq} } }}}\\
&\qquad  - \sum_{k = 1}^q {\frac{1}{{\prod_{j = 0}^{m - 1} {F_{2nk + 2nN + 2jnq} } }}}\,,
\end{split}
\end{equation}
from which Theorem~\ref{thm.un8i8oz} follows in the limit as $N$ approaches infinity.
\end{proof}
\begin{thm}\label{thm.m05lv0w}
If $m$, $n$ and $q$ are positive integers such that $q$ is even \underline{or} $mnq$ is odd , then
\[
\sum_{k = 1}^\infty  {\frac{(-1)^{k-1}L_{2nk + mnq} }{{\prod_{j = 0}^m {F_{2nk + 2jnq} } }}}  = \frac{1}{{F_{mnq} }}\sum_{k = 1}^q {\frac{(-1)^{k-1}}{{\prod_{j = 0}^{m - 1} {F_{2nk + 2jnq} } }}}\,. 
\]
\end{thm}
In particular,
\begin{equation}\label{equ.rdxyneo}
\sum_{k = 1}^\infty  {\frac{{(-1)^{k-1}L_{2nk + nq} }}{{F_{2nk} F_{2nk + 2nq} }}}  = \frac{1}{{F_{nq} }}\sum_{k = 1}^q {\frac{(-1)^{k-1}}{{F_{2nk} }}}\,,\quad\mbox{$q$ even \underline{or} $nq$ odd}\,.
\end{equation}
The alternating summation identity here, valid for $q$ even or $mnq$ odd, is
\begin{equation}
\begin{split}
F_{mnq} \sum_{k = 1}^N {\frac{{(-1)^{k-1}L_{2nk + mnq} }}{{\prod_{j = 0}^m {F_{2nk + 2jnq} } }}}  &= \sum_{k = 1}^q {\frac{(-1)^{k-1}}{{\prod_{j = 0}^{m - 1} {F_{2nk + 2jnq} } }}}\\
&\qquad  +(-1)^{N-1} \sum_{k = 1}^q {\frac{(-1)^{k-1}}{{\prod_{j = 0}^{m - 1} {F_{2nk + 2nN + 2jnq} } }}}\,.
\end{split}
\end{equation}
\subsection{Sums with $L_{2nk}L_{2nk+2nq}L_{2nk+4nq}L_{2nk+6nq}\cdots L_{2nk+2mnq}$ in the denominator}\label{sec.esv6wcx}
The results in this section are derived from identities~\eqref{equ.epvyp3u}. The proofs are identical to those in section~\ref{sec.ffn2hfi} and are therefore omitted.
\begin{thm}
If $m$, $n$ and $q$ are positive \underline{odd} integers, then
\[
\sum_{k = 1}^\infty  {\frac{{L_{2nk + mnq} }}{{\prod_{j = 0}^m {L_{2nk + 2jnq} } }}}  = \frac{1}{{L_{mnq} }}\sum_{k = 1}^q {\frac{1}{{\prod_{j = 0}^{m - 1} {L_{2nk + 2jnq} } }}}\,. 
\]
\end{thm}
In particular,
\begin{equation}
\sum_{k = 1}^\infty  {\frac{{L_{2nk + nq} }}{{L_{2nk} L_{2nk + 2nq} }}}  = \frac{1}{{L_{nq} }}\sum_{k = 1}^q {\frac{1}{{L_{2nk} }}}\,,\quad\mbox{$nq$ odd}\,,
\end{equation}
\begin{equation}
\sum_{k = 1}^\infty  {\frac{{L_{2nk + 3nq} }}{{L_{2nk} L_{2nk + 2nq} L_{2nk + 4nq} L_{2nk + 6nq} }}}  = \frac{1}{{L_{3nq} }}\sum_{k = 1}^q {\frac{1}{{L_{2nk} L_{2nk + 2nq} L_{2nk + 4nq} }}}\,.
\end{equation}
The finite summation identity is
\begin{equation}
\begin{split}
L_{mnq} \sum_{k = 1}^N {\frac{{L_{2nk + mnq} }}{{\prod_{j = 0}^m {L_{2nk + 2jnq} } }}}  &= \sum_{k = 1}^q {\frac{1}{{\prod_{j = 0}^{m - 1} {L_{2nk + 2jnq} } }}}\qquad\qquad\mbox{$mnq$ odd}\\
&\qquad  - \sum_{k = 1}^q {\frac{1}{{\prod_{j = 0}^{m - 1} {L_{2nk + 2nN + 2jnq} } }}}\,.
\end{split}
\end{equation}
\begin{thm}
If $m$, $n$ and $q$ are positive integers such that $q$ is odd and $mn$ is even , then
\[
\sum_{k = 1}^\infty  {\frac{(-1)^{k-1}L_{2nk + mnq} }{{\prod_{j = 0}^m {L_{2nk + 2jnq} } }}}  = \frac{1}{{L_{mnq} }}\sum_{k = 1}^q {\frac{(-1)^{k-1}}{{\prod_{j = 0}^{m - 1} {L_{2nk + 2jnq} } }}}\,. 
\]
\end{thm}
In particular,
\begin{equation}
\sum_{k = 1}^\infty  {\frac{{(-1)^{k-1}L_{2nk + nq} }}{{L_{2nk} L_{2nk + 2nq} }}}  = \frac{1}{{L_{nq} }}\sum_{k = 1}^q {\frac{(-1)^{k-1}}{{L_{2nk} }}}\,,\qquad \mbox{$q$ odd, $n$ even}\,.
\end{equation}
The alternating finite summation identity is
\begin{equation}
\begin{split}
L_{mnq} \sum_{k = 1}^N {\frac{{(-1)^{k-1}L_{2nk + mnq} }}{{\prod_{j = 0}^m {L_{2nk + 2jnq} } }}}  &= \sum_{k = 1}^q {\frac{(-1)^{k-1}}{{\prod_{j = 0}^{m - 1} {L_{2nk + 2jnq} } }}}\qquad\qquad\mbox{$q$ odd, $mn$ even}\\
&\qquad  +(-1)^{N-1} \sum_{k = 1}^q {\frac{(-1)^{k-1}}{{\prod_{j = 0}^{m - 1} {L_{2nk + 2nN + 2jnq} } }}}\,.
\end{split}
\end{equation}
\begin{thm}
If $m$, $n$ and $q$ are positive integers such that $mnq$ is even, then
\[
\sum_{k = 1}^\infty  {\frac{{F_{2nk + mnq} }}{{\prod_{j = 0}^m {L_{2nk + 2jnq} } }}}  = \frac{1}{{5F_{mnq} }}\sum_{k = 1}^q {\frac{1}{{\prod_{j = 0}^{m - 1} {L_{2nk + 2jnq} } }}}\,. 
\]

\end{thm}
In particular,
\begin{equation}
\sum_{k = 1}^\infty  {\frac{{F_{2nk + nq} }}{{L_{2nk} L_{2nk + 2nq} }}}  = \frac{1}{{5F_{nq} }}\sum_{k = 1}^q {\frac{1}{{L_{2nk} }}}\,,\qquad\mbox{$nq$ even}\,,
\end{equation}
\begin{equation}
\sum_{k = 1}^\infty  {\frac{{F_{2nk + 3nq} }}{{L_{2nk} L_{2nk + 2nq} L_{2nk + 4nq} L_{2nk + 6nq} }}}  = \frac{1}{{5F_{3nq} }}\sum_{k = 1}^q {\frac{1}{{L_{2nk} L_{2nk + 2nq} L_{2nk + 4nq} }}}\,.
\end{equation}
The finite summation identity is
\begin{equation}
\begin{split}
5F_{mnq} \sum_{k = 1}^N {\frac{{F_{2nk + mnq} }}{{\prod_{j = 0}^m {L_{2nk + 2jnq} } }}}  &= \sum_{k = 1}^q {\frac{1}{{\prod_{j = 0}^{m - 1} {L_{2nk + 2jnq} } }}}\qquad\qquad\mbox{$mnq$ even}\\
&\qquad  - \sum_{k = 1}^q {\frac{1}{{\prod_{j = 0}^{m - 1} {L_{2nk + 2nN + 2jnq} } }}}\,,
\end{split}
\end{equation}
\begin{thm}
If $m$, $n$ and $q$ are positive integers such that $q$ is even \underline{or} $mnq$ is odd , then
\[
\sum_{k = 1}^\infty  {\frac{(-1)^{k-1}F_{2nk + mnq} }{{\prod_{j = 0}^m {L_{2nk + 2jnq} } }}}  = \frac{1}{{5F_{mnq} }}\sum_{k = 1}^q {\frac{(-1)^{k-1}}{{\prod_{j = 0}^{m - 1} {L_{2nk + 2jnq} } }}}\,. 
\]
\end{thm}
In particular,
\begin{equation}
\sum_{k = 1}^\infty  {\frac{{(-1)^{k-1}F_{2nk + nq} }}{{L_{2nk} F_{2nk + 2nq} }}}  = \frac{1}{{5F_{nq} }}\sum_{k = 1}^q {\frac{(-1)^{k-1}}{{L_{2nk} }}}\,,\quad\mbox{$q$ even \underline{or} $nq$ odd}\,.
\end{equation}
The alternating summation identity here, valid for $q$ even or $mnq$ odd, is
\begin{equation}
\begin{split}
5F_{mnq} \sum_{k = 1}^N {\frac{{(-1)^{k-1}F_{2nk + mnq} }}{{\prod_{j = 0}^m {L_{2nk + 2jnq} } }}}  &= \sum_{k = 1}^q {\frac{(-1)^{k-1}}{{\prod_{j = 0}^{m - 1} {L_{2nk + 2jnq} } }}}\\
&\qquad  +(-1)^{N-1} \sum_{k = 1}^q {\frac{(-1)^{k-1}}{{\prod_{j = 0}^{m - 1} {L_{2nk + 2nN + 2jnq} } }}}\,.
\end{split}
\end{equation}

\subsection{Sums with $F_{nk}^2F_{nk+nq}^2\cdots F_{nk+mnq-nq}^2F_{nk+mnq}F_{nk+mnq+nq}^2\cdots F_{nk+2mnq}^2$ or $F_{nk}^2F_{nk+nq}^2\cdots F_{nk+mnq}^2$ in the denominator}
\begin{thm}\label{thm.kvn1tje}
If $m$, $n$ and $q$ are positive integers such that $mnq$ is even, then
\[
\sum_{k = 1}^\infty  {\left[ {\frac{{F_{2nk + mnq} }}{{\prod_{j = 0}^m {F_{nk + jnq}^2 } }}} \right]}  = \frac{1}{{F_{mnq} }}\sum_{k = 1}^q {\frac{1}{{\prod_{j = 0}^{m - 1} {F_{nk + jnq}^2 } }}}\,. 
\]
\end{thm}
Explicitly,
\[
\sum_{k = 1}^\infty  {\frac{{F_{2nk + mnq} }}{{F_{nk}^2 F_{nk + nq}^2  \cdots F_{nk + mnq}^2 }}}  = \frac{1}{{F_{mnq} }}\sum_{k = 1}^q {\frac{1}{{F_{nk}^2 F_{nk + nq}^2  \cdots F_{nk + (m - 1)nq}^2 }}}\,. 
\]
Examples include:
\begin{equation}
\begin{split}
&\mbox{At $m=1$}\\
&\sum_{k = 1}^\infty  {\frac{{F_{2nk + nq} }}{{F_{nk}^2 F_{nk + nq}^2 }}}  = \frac{1}{{F_{nq} }}\sum_{k = 1}^q {\frac{1}{{F_{nk}^2 }}},\quad \mbox{$nq$ even}\,.
\end{split}
\end{equation}
\begin{equation}
\begin{split}
&\mbox{At $(m,n,q)=(1,1,2)$ and $(m,n,q)=(1,2,1)$}:\\
&\sum_{k = 1}^\infty  {\frac{{F_{2k + 2} }}{{F_k^2 F_{k + 2}^2 }}}  = 2,\quad\sum_{k = 1}^\infty  {\frac{{F_{4k + 2} }}{{F_{2k}^2 F_{2k + 2}^2 }}}  = 1\,.
\end{split}
\end{equation}

\begin{equation}
\begin{split}
&\mbox{At $(m,n,q)=(3,2,2)$}:\\
&\sum_{k = 1}^\infty  {\frac{{F_{4k + 12} }}{{F_{2k}^2 F_{2k + 4}^2 F_{2k + 8}^2 F_{2k + 12}^2 }}}  = \frac{1288981}{35850395750400}\,.
\end{split}
\end{equation}

\begin{cor}\label{thm.rh1ujkm}
If $m$, $n$ and $q$ are positive integers, then
\[
\sum_{k = 1}^\infty  {\left[ {\frac{{L_{nk + mnq} }}{{F_{nk + mnq} \prod_{j = 0}^{m - 1} {F_{nk + jnq}^2 } \prod_{j = m + 1}^{2m} {F_{nk + jnq}^2 } }}} \right]}  = \frac{1}{{F_{2mnq} }}\sum_{k = 1}^q {\left[ {\prod_{j = 0}^{2m - 1} {\frac{1}{{F_{nk + jnq}^2 }}} } \right]}\,. 
\]

\end{cor}
Explicitly, Corollary~\ref{thm.rh1ujkm} is
\[
\begin{split}
&\sum_{k = 1}^\infty  {\frac{{L_{nk + mnq} }}{{F_{nk}^2 F_{nk + nq}^2  \cdots F_{nk + (m - 1)nq}^2 F_{nk + mnq} F_{nk + (m + 1)nq}^2  \cdots F_{nk + 2mnq}^2 }}}\\
&\qquad\qquad\qquad= \frac{1}{{F_{2mnq} }}\sum_{k = 1}^q {\frac{1}{{F_{nk}^2 F_{nk + nq}^2  \cdots F_{nk + (2m - 1)nq}^2 }}}\,.
\end{split}
\]
Below are a couple of examples:
\begin{equation}
\begin{split}
&\mbox{At $(m,n,q)=(1,1,1)$ and $(m,n,q)=(2,1,1)$}:\\
&\sum_{k = 1}^\infty  {\frac{{L_{k + 1} }}{{F_k^2 F_{k + 1} F_{k + 2}^2 }}}  = 1,\quad\sum_{k = 1}^\infty  {\frac{{L_{k + 2} }}{{F_k^2 F_{k + 1}^2 F_{k + 2} F_{k + 3}^2 F_{k + 4}^2 }}}  = \frac{1}{{108}}
\,.
\end{split}
\end{equation}
\begin{equation}
\begin{split}
&\mbox{At $(m,n,q)=(3,2,2)$}:\\
&\sum_{k = 1}^\infty  {\frac{{L_{2k + 12} }}{{F_{2k}^2 F_{2k + 4}^2 F_{2k + 8}^2 F_{2k + 12} F_{2k + 16}^2 F_{2k + 20}^2 F_{2k + 24}^2 }}}\\
&= \frac{636693716175181614930457}{1701394375843622618689225675379000792710492054565683200}\,.
\end{split}
\end{equation}

\subsubsection*{Proof of Theorem~\ref{thm.kvn1tje}}
By making appropriate choices for the indices $u$ and $v$ in the identity~\eqref{equ.ded0k7c}, it is straightforward to establish the following identity:
\begin{equation}\label{equ.g4pbkha}
\frac{{F_{mnq} F_{2nk + mnq} }}{{F_{nk}^2 F_{nk + mnq}^2 }} = \frac{1}{{F_{nk}^2 }} - \frac{{( - 1)^{mnq} }}{{F_{nk + mnq}^2 }}\,,
\end{equation}
so that
\begin{equation}\label{equ.vcnyyet}
\frac{{F_{mnq} F_{2nk + mnq} }}{{F_{nk}^2 F_{nk + mnq}^2 }} = \frac{1}{{F_{nk}^2 }} - \frac{1 }{F_{nk + mnq}^2 },\quad\mbox{$mnq$ even}
\end{equation}
and
\begin{equation}\label{equ.werykpm}
\frac{{F_{mnq} F_{2nk + mnq} }}{{F_{nk}^2 F_{nk + mnq}^2 }} = \frac{1}{{F_{nk}^2 }} + \frac{1}{{F_{nk + mnq}^2 }},\quad\mbox{$mnq$ odd}\,.
\end{equation}
From~\eqref{equ.vcnyyet}, with $f(k)=1/F_k^2$ in Lemma~\ref{finall}, we have the finite summation identity
\begin{equation}
\begin{split}
\sum_{k = 1}^N {\frac{{F_{2nk + mnq} }}{{\prod_{j = 0}^m {F_{nk + jnq}^2 } }}}  &= \frac{1}{{F_{mnq} }}\sum_{k = 1}^q {\frac{1}{{\prod_{j = 0}^{m - 1} {F_{nk + jnq}^2 } }}}\\
&\qquad - \frac{1}{{F_{mnq} }}\sum_{k = 1}^q {\frac{1}{{\prod_{j = 0}^{m - 1} {F_{nk + nN + jnq}^2 } }}},\quad\mbox{$mnq$ even}\,.
\end{split}
\end{equation}
As $N$ approaches infinity, we have Theorem~\ref{thm.kvn1tje}, while specifically requiring $m$ to be even gives Corollary~\ref{thm.rh1ujkm}.

\begin{thm}\label{thm.hblrah9}
If $m$, $n$ and $q$ are integers such that $q$ is even \underline{or} $mnq$ is odd, then
\[
\sum_{k = 1}^\infty  {\left[ {\frac{{( - 1)^{k - 1} F_{2nk + mnq} }}{{\prod_{j = 0}^m {F_{nk + jnq}^2 } }}} \right]}  = \frac{1}{{F_{mnq} }}\sum_{k = 1}^q {\left[ {\frac{{( - 1)^{k - 1} }}{{\prod_{j = 0}^{m - 1} {F_{nk + jnq}^2 } }}} \right]}\,. 
\]
\end{thm}
\begin{proof}
If $q$ is even, then the statement of the theorem follows from~\eqref{equ.vcnyyet} and identity~\eqref{infallalt} with the upper sign or from~\eqref{equ.werykpm} and identity~\eqref{infallalt} with the lower sign if $mnq$ is odd.
\end{proof}
In particular, we have
\begin{equation}\label{equ.l14z6cr}
\sum_{k = 1}^\infty  {\frac{{( - 1)^{k - 1} F_{2nk + nq} }}{{F_{nk}^2 F_{nk + nq}^2 }}}  = \frac{1}{{F_{nq} }}\sum_{k = 1}^q {\frac{{( - 1)^{k - 1} }}{{F_{nk}^2 }}},\quad \mbox{$q$ even or $nq$ odd}\,,
\end{equation}
which generalizes Brousseau's result (Formula~(6) of~\cite{brousseau1}, also derived by Melham in reference~\cite{melham16} as a special case of a more general result). Brousseau's formula~(6) corresponds to $n=1$ in~\eqref{equ.l14z6cr}.  Brousseau's formula~(15) in~\cite{brousseau1} is also contained in the identity~\eqref{equ.l14z6cr} above at \mbox{$n=3,\,q=1$}.

\bigskip

More examples from Theorem~\ref{thm.hblrah9}:
\begin{equation}
\sum_{k = 1}^\infty  {\frac{{( - 1)^{k - 1} F_{2nk + 3nq} }}{{F_{nk}^2 F_{nk + nq}^2 F_{nk + 2nq}^2 F_{nk + 3nq}^2 }}}  = \frac{1}{{F_{3nq} }}\sum_{k = 1}^q {\frac{{( - 1)^{k - 1} }}{{F_{nk}^2 F_{nk + nq}^2 F_{nk + 2nq}^2 }}},\quad \mbox{$q$ even or $nq$ odd}\,.
\end{equation}

\begin{cor}\label{thm.p79aeki}
If $q$ is a positive \underline{even} integer, then
\[
\sum_{k = 1}^\infty  {\left[ {\frac{{(-1)^{k-1}L_{nk + mnq} }}{{F_{nk + mnq} \prod_{j = 0}^{m - 1} {F_{nk + jnq}^2 } \prod_{j = m + 1}^{2m} {F_{nk + jnq}^2 } }}} \right]}  = \frac{1}{{F_{2mnq} }}\sum_{k = 1}^q {(-1)^{k-1}\left[ {\prod_{j = 0}^{2m - 1} {\frac{1}{{F_{nk + jnq}^2 }}} } \right]}\,. 
\]

\end{cor}
In particular:
\begin{equation}
\sum_{k = 1}^\infty  {\frac{{( - 1)^{k - 1} L_{k + q} }}{{F_k^2 F_{k + q} F_{k + 2q}^2 }}}  = \frac{1}{{F_{2q} }}\sum_{k = 1}^q {\frac{{( - 1)^{k - 1} }}{{F_k^2 F_{k + q}^2 }}},\quad\mbox{$q$ even}\,.
\end{equation}
We note that Theorem~\ref{thm.kvn1tje}, Theorem~\ref{thm.hblrah9}, Corollary~\ref{thm.rh1ujkm} and Corollary~\ref{thm.p79aeki} correspond to setting $p=0$ in Theorem \ref{thm.byyca8f} and Theorem \ref{thm.xbwcpz1} of section~\ref{sec.cp47u43}.
\begin{thm}\label{thm.xlspejf}
If $m$, $n$ and $q$ are positive integers, then
\[
\sum_{k = 1}^\infty  {\left[ {\frac{{( - 1)^{nk - 1} F_{2nk + mnq} \prod_{j = 1}^{m - 1} {L_{nk + jnq}^2 } }}{{\prod_{j = 0}^m {F_{nk + jnq}^2 } }}} \right]}  = \frac{{5^m q}}{{4F_{mnq} }} - \frac{1}{{4F_{mnq} }}\sum_{k = 1}^q {\left[ {\prod_{j = 0}^{m - 1} {\frac{{L_{nk + jnq}^2 }}{{F_{nk + jnq}^2 }}} } \right]}\,. 
\]

\end{thm}
In particular
\begin{equation}
\sum_{k = 1}^\infty  {\frac{{( - 1)^{nk - 1} F_{2nk + nq} }}{{F_{nk}^2 F_{nk + nq}^2 }}}  = \frac{{5q}}{{4F_{nq} }} - \frac{1}{{4F_{nq} }}\sum_{k = 1}^q {\frac{{L_{nk}^2 }}{{F_{nk}^2 }}}\,.
\end{equation}

\begin{thm}\label{thm.wxp5pic}
If $n$ and $q$ are positive integers such that $n$ is odd and $q$ is even, then
\[
\sum_{k = 1}^\infty  {\left[ {\frac{{F_{2nk + mnq} \prod_{j = 1}^{m - 1} {L_{nk + jnq}^2 } }}{{\prod_{j = 0}^m {F_{nk + jnq}^2 } }}} \right]}  = \frac{1}{{4F_{mnq} }}\sum_{k = 1}^q {\left[ {( - 1)^k \prod_{j = 0}^{m - 1} {\frac{{L_{nk + jnq}^2 }}{{F_{nk + jnq}^2 }}} } \right]}\,. 
\]

\end{thm}

\subsubsection*{Proof of Theorem~\ref{thm.xlspejf} and Theorem~\ref{thm.wxp5pic}}
Multiplying identity~\eqref{equ.ejrnkwy} and identity~\eqref{equ.moytk3x} and choosing $u$ and $v$ judiciously, it is easy to establish that the following identity holds for positive integers $m$, $n$, $q$ and $k$:
\begin{equation}\label{equ.oeu2ljq}
( - 1)^{nk - 1} \frac{{4F_{mnq} F_{2nk + mnq} }}{{F_{nk}^2 F_{nk + mnq}^2 }} = \frac{{L_{nk + mnq}^2 }}{{F_{nk + mnq}^2 }} - \frac{{L_{nk}^2 }}{{F_{nk}^2 }}\,.
\end{equation}
From identity~\eqref{equ.oeu2ljq} and $f(k)=L_k^2/F_k^2$ in Lemma~\ref{finall} we have the finite summation formula:
\begin{equation}
\begin{split}
&4F_{mnq} \sum_{k = 1}^N {\frac{{( - 1)^{nk - 1} F_{2nk + mnq} \prod_{j = 1}^{m - 1} {L_{nk + jnq}^2 } }}{{\prod_{j = 0}^m {F_{nk + jnq}^2 } }}}\\
&\qquad  = \sum_{k = 1}^q {\prod_{j = 0}^{m - 1} {\frac{{L_{nk + nN + jnq}^2 }}{{F_{nk + nN + jnq}^2 }}} }  - \sum_{k = 1}^q {\prod_{j = 0}^{m - 1} {\frac{{L_{nk + jnq}^2 }}{{F_{nk + jnq}^2 }}} }\,,
\end{split}
\end{equation}
from which Theorem~\ref{thm.xlspejf} follows in the limit $N$ approaches infinity. Theorem~\ref{thm.wxp5pic} follows from identity~\eqref{infallalt}.

\subsection{Sums with $L_{nk}^2L_{nk+nq}^2\cdots L_{nk+mnq-nq}^2L_{nk+mnq}L_{nk+mnq+nq}^2\cdots L_{nk+2mnq}^2$ or $L_{nk}^2L_{nk+nq}^2\cdots L_{nk+mnq}^2$ in the denominator}
The theorems in this section are the Lucas versions of those of the previous section. We omit their proofs. The basic identity is
\begin{equation}\label{equ.o1vymkb}
\frac{{5F_{mnq} F_{2nk + mnq} }}{{L_{nk}^2 L_{nk + mnq}^2 }} = \frac{1}{{L_{nk}^2 }} - \frac{{( - 1)^{mnq} }}{{L_{nk + mnq}^2 }}\,,
\end{equation}
which follows from the identity~\eqref{equ.wodrq78}.
\begin{thm}\label{thm.wfq1axk}
If $m$, $n$ and $q$ are positive integers such that $mnq$ is even, then
\[
\sum_{k = 1}^\infty  {\left[ {\frac{{F_{2nk + mnq} }}{{\prod_{j = 0}^m {L_{nk + jnq}^2 } }}} \right]}  = \frac{1}{{5F_{mnq} }}\sum_{k = 1}^q {\left[ {\prod_{j = 0}^{m - 1} {\frac{1}{{L_{nk + jnq}^2 }}} } \right]}\,. 
\]
\end{thm}
Explicitly,
\[
\sum_{k = 1}^\infty  {\frac{{F_{2nk + mnq} }}{{L_{nk}^2 L_{nk + nq}^2  \cdots L_{nk + mnq}^2 }}}  = \frac{1}{{5F_{mnq} }}\sum_{k = 1}^q {\frac{1}{{L_{nk}^2 L_{nk + nq}^2  \cdots L_{nk + (m - 1)nq}^2 }}}\,. 
\]
\begin{cor}\label{thm.v07bx6w}
If $m$, $n$ and $q$ are positive integers, then
\[
\sum_{k = 1}^\infty  {\left[ {\frac{{F_{nk + mnq} }}{{L_{nk + mnq} \prod_{j = 0}^{m - 1} {L_{nk + jnq}^2 } \prod_{j = m + 1}^{2m} {L_{nk + jnq}^2 } }}} \right]}  = \frac{1}{{5F_{2mnq} }}\sum_{k = 1}^q {\left[ {\prod_{j = 0}^{2m - 1} {\frac{1}{{L_{nk + jnq}^2 }}} } \right]}\,. 
\]

\end{cor}
\begin{thm}\label{thm.t9u45bs}
If $m$, $n$ and $q$ are positive integers such that $q$ is even \underline{or} $mnq$ is odd, then
\[
\sum_{k = 1}^\infty  {\left[ {\frac{{( - 1)^{k - 1} F_{2nk + mnq} }}{{\prod_{j = 0}^m {L_{nk + jnq}^2 } }}} \right]}  = \frac{1}{{5F_{mnq} }}\sum_{k = 1}^q {\left[ {\frac{{( - 1)^{k - 1} }}{{\prod_{j = 0}^{m - 1} {L_{nk + jnq}^2 } }}} \right]}\,. 
\]

\end{thm}
In particular, we have
\begin{equation}\label{equ.gk5b94r}
\sum_{k = 1}^\infty  {\frac{{( - 1)^{k - 1} F_{2nk + nq} }}{{L_{nk}^2 L_{nk + nq}^2 }}}  = \frac{1}{{5F_{nq} }}\sum_{k = 1}^q {\frac{{( - 1)^{k - 1} }}{{L_{nk}^2 }}},\quad \mbox{$q$ even or $nq$ odd}\,,
\end{equation}
\begin{cor}\label{thm.uho7fyr}
If $q$ is a positive \underline{even} integer, then
\[
\sum_{k = 1}^\infty  {\left[ {\frac{{(-1)^{k-1}F_{nk + mnq} }}{{L_{nk + mnq} \prod_{j = 0}^{m - 1} {L_{nk + jnq}^2 } \prod_{j = m + 1}^{2m} {L_{nk + jnq}^2 } }}} \right]}  = \frac{1}{{5F_{2mnq} }}\sum_{k = 1}^q {\left[ {\frac{{( - 1)^{k - 1} }}{{\prod_{j = 0}^{2m - 1} {L_{nk + jnq}^2 } }}} \right]}\,. 
\]

\end{cor}
In particular:
\begin{equation}
\sum_{k = 1}^\infty  {\frac{{( - 1)^{k - 1} F_{k + q} }}{{L_k^2 L_{k + q} L_{k + 2q}^2 }}}  = \frac{1}{{5F_{2q} }}\sum_{k = 1}^q {\frac{{( - 1)^{k - 1} }}{{L_k^2 L_{k + q}^2 }}},\quad\mbox{$q$ even}\,.
\end{equation}

\begin{thm}\label{thm.fw565ya}
If $m$, $n$ and $q$ are positive integers, then
\[
\sum_{k = 1}^\infty  {\left[ {\frac{{( - 1)^{nk - 1} F_{2nk + mnq} \prod_{j = 1}^{m - 1} {F_{nk + jnq}^2 } }}{{\prod_{j = 0}^m {L_{nk + jnq}^2 } }}} \right]}  =  \frac{1}{{4F_{mnq} }}\sum_{k = 1}^q {\left[ {\prod_{j = 0}^{m - 1} {\frac{{F_{nk + jnq}^2 }}{{L_{nk + jnq}^2 }}} } \right]}-\frac{{1}}{{4F_{mnq} }}\frac q{5^m}\,. 
\]

\end{thm}
In particular
\begin{equation}
\sum_{k = 1}^\infty  {\frac{{( - 1)^{nk - 1} F_{2nk + nq} }}{{L_{nk}^2 L_{nk + nq}^2 }}}  = \frac{1}{{4F_{nq} }}\sum_{k = 1}^q {\frac{{F_{nk}^2 }}{{L_{nk}^2 }}}-\frac{q}{{20F_{nq} }}\,.
\end{equation}

\begin{thm}\label{thm.qwgs5xu}
If $n$ and $q$ are positive integers such that $n$ is odd and $q$ is even, then
\[
\sum_{k = 1}^\infty  {\left[ {\frac{{F_{2nk + mnq} \prod_{j = 1}^{m - 1} {F_{nk + jnq}^2 } }}{{\prod_{j = 0}^m {L_{nk + jnq}^2 } }}} \right]}  = \frac{1}{{4F_{mnq} }}\sum_{k = 1}^q {\left[ {( - 1)^{k-1} \prod_{j = 0}^{m - 1} {\frac{{F_{nk + jnq}^2 }}{{L_{nk + jnq}^2 }}} } \right]}\,. 
\]

\end{thm}
In particular
\begin{equation}
\sum_{k = 1}^\infty  {\frac{{ F_{2nk + nq} }}{{L_{nk}^2 L_{nk + nq}^2 }}}  = \frac{1}{{4F_{nq} }}\sum_{k = 1}^q {( - 1)^{k - 1}\frac{{F_{nk}^2 }}{{L_{nk}^2 }}},\quad\mbox{$n$ odd, $q$ even}\,.
\end{equation}

\subsection{Sums with $F_{nk}F_{nk+np}F_{nk+nq}F_{nk+nq+np}F_{nk+2nq}F_{nk+2nq+np}$$\cdots$ $F_{nk+mnq}F_{nk+mnq+np}$ in the denominator}\label{sec.cp47u43}
\begin{thm}\label{thm.byyca8f}
If $m$, $n$, $q$ are positive integers such that $mnq$ is even; and $p$ is a non-negative integer, then
\[
\sum_{k = 1}^\infty  {\frac{{F_{2nk + mnq + np} }}{{\prod_{j = 0}^m {F_{nk + jnq} F_{nk + jnq + np} } }}}  = \frac{1}{{F_{mnq} }}\sum_{k = 1}^q {\frac{1}{{\prod_{j = 0}^{m - 1} {F_{nk + jnq} F_{nk + jnq + np} } }}}\,. 
\]
\end{thm}
\begin{thm}\label{thm.xbwcpz1}
If $m$, $n$, $q$ are positive integers such that $mnq$ is odd or $q$ is even; and $p$ is a non-negative integer, then
\[
\sum_{k = 1}^\infty  {\frac{{( - 1)^{k - 1} F_{2nk + mnq + np} }}{{\prod_{j = 0}^m {F_{nk + jnq} F_{nk + jnq + np} } }}}  = \frac{1}{{F_{mnq} }}\sum_{k = 1}^q {\frac{{( - 1)^{k - 1} }}{{\prod_{j = 0}^{m - 1} {F_{nk + jnq} F_{nk + jnq + np} } }}}\,. 
\]
\end{thm}
\subsubsection*{Proof of Theorem \ref{thm.byyca8f} and Theorem \ref{thm.xbwcpz1}}
By dividing through the identity~\eqref{equ.cqtsjoj} by $F_uF_{u+p}F_{u+v}F_{u+v+p}$ and setting $u=nk$ and $v=mnq$, the following identity is established
\begin{equation}\label{equ.chzhm4b}
\frac{{F_{mnq} F_{2nk + mnq + np} }}{{F_{nk} F_{nk + np} F_{nk + mnq} F_{nk + mnq + np} }} = \frac{1}{{F_{nk} F_{nk + np} }} + \frac{{( - 1)^{mnq + 1} }}{{F_{nk + mnq} F_{nk + mnq + np} }}\,,
\end{equation}
so that
\begin{equation}\label{equ.wz80v2j}
\begin{split}
&\frac{{F_{mnq} F_{2nk + mnq + np} }}{{F_{nk} F_{nk + np} F_{nk + mnq} F_{nk + mnq + np} }}\\ 
&\qquad= \frac{1}{{F_{nk} F_{nk + np} }} - \frac{1}{{F_{nk + mnq} F_{nk + mnq + np} }},\quad\mbox{$mnq$ even}
\end{split}
\end{equation}
and
\begin{equation}\label{equ.hitwupq}
\begin{split}
&\frac{{F_{mnq} F_{2nk + mnq + np} }}{{F_{nk} F_{nk + np} F_{nk + mnq} F_{nk + mnq + np} }}\\
&\qquad = \frac{1}{{F_{nk} F_{nk + np} }} + \frac{1}{{F_{nk + mnq} F_{nk + mnq + np} }},\quad\mbox{$mnq$ odd}\,.
\end{split}
\end{equation}
From~\eqref{equ.wz80v2j} and $f(k)=1/(F_kF_{k+np})$ in Lemma~\ref{infall} we obtain the finite summation identity
\begin{equation}
\begin{split}
&F_{mnq} \sum_{k = 1}^N {\left[ {\frac{{F_{2nk + mnq + np} }}{{\prod_{j = 0}^m {F_{nk + jnq} F_{nk + jnq + np} } }}} \right]}\\ 
&\qquad = \sum_{k = 1}^q {\left[ {\frac{1}{{\prod_{j = 0}^{m - 1} {F_{nk + jnq} F_{nk + jnq + np} } }}} \right]}\\
&\qquad\quad  - \sum_{k = 1}^q {\left[ {\frac{1}{{\prod_{j = 0}^{m-1} {F_{nk + nN + jnq} F_{nk + nN + jnq + np} } }}} \right]}\,, 
\end{split}
\end{equation}
from which Theorem~\ref{thm.byyca8f} follows in the limit as $N$ approaches infinity. If $q$ is even, then the statement of Theorem~\ref{thm.xbwcpz1} follows from~\eqref{equ.wz80v2j} and identity~\eqref{infallalt} with the upper sign or from~\eqref{equ.hitwupq} and identity~\eqref{infallalt} with the lower sign if $mnq$ is odd.

\subsection{Sums with $L_{nk}L_{nk+np}L_{nk+nq}L_{nk+nq+np}L_{nk+2nq}L_{nk+2nq+np}$$\cdots$ $L_{nk+mnq}L_{nk+mnq+np}$ in the denominator}
The results in this section are the Lucas versions of the results in the preceding section. The basic identities are
\begin{equation}
\frac{{5F_{mnq} F_{2nk + mnq + np} }}{{\prod_{j = 0}^m {L_{nk + jnq} L_{nk + jnq + np} } }} = \frac{1}{{\prod_{j = 0}^{m - 1} {L_{nk + jnq} L_{nk + jnq + np} } }} + \frac{{( - 1)^{mnq + 1} }}{{\prod_{j = 1}^m {L_{nk + jnq} L_{nk + jnq + np} } }}\,,
\end{equation}
and, if $mnq$ is even
\begin{equation}
\begin{split}
&5F_{mnq} \sum_{k = 1}^N {\left[ {\frac{{F_{2nk + mnq + np} }}{{\prod_{j = 0}^m {L_{nk + jnq} L_{nk + jnq + np} } }}} \right]}\\ 
&\qquad\quad = \sum_{k = 1}^q {\left[ {\frac{1}{{\prod_{j = 0}^{m - 1} {L_{nk + jnq} L_{nk + jnq + np} } }}} \right]}\\
&\qquad\qquad  - \sum_{k = 1}^q {\left[ {\frac{1}{{\prod_{j = 0}^{m-1} {L_{nk + nN + jnq} L_{nk + nN + jnq + np} } }}} \right]}\,, 
\end{split}
\end{equation}
while if $mnq$ is odd or $q$ is even:
\begin{equation}
\begin{split}
&5F_{mnq} \sum_{k = 1}^N {\left[ {\frac{{(-1)^{k-1}F_{2nk + mnq + np} }}{{\prod_{j = 0}^m {L_{nk + jnq} L_{nk + jnq + np} } }}} \right]}\\ 
&\qquad\quad = \sum_{k = 1}^q {\left[ {\frac{(-1)^{k-1}}{{\prod_{j = 0}^{m - 1} {L_{nk + jnq} L_{nk + jnq + np} } }}} \right]}\\
&\qquad\qquad  +(-1)^{N-1} \sum_{k = 1}^q {\left[ {\frac{(-1)^{k-1}}{{\prod_{j = 0}^{m-1} {L_{nk + nN + jnq} L_{nk + nN + jnq + np} } }}} \right]}\,. 
\end{split}
\end{equation}
\begin{thm}\label{thm.ckvmm33}
If $m$, $n$, $q$ are positive integers such that $mnq$ is even; and $p$ is a non-negative integer, then
\[
\sum_{k = 1}^\infty  {\frac{{F_{2nk + mnq + np} }}{{\prod_{j = 0}^m {L_{nk + jnq} L_{nk + jnq + np} } }}}  = \frac{1}{{5F_{mnq} }}\sum_{k = 1}^q {\frac{1}{{\prod_{j = 0}^{m - 1} {L_{nk + jnq} L_{nk + jnq + np} } }}}\,. 
\]

\end{thm}\label{thm.v7eumk9}
\begin{thm}
If $m$, $n$, $q$ are positive integers such that $mnq$ is odd or $q$ is even; and $p$ is a non-negative integer, then
\[
\sum_{k = 1}^\infty  {\frac{{( - 1)^{k - 1} F_{2nk + mnq + np} }}{{\prod_{j = 0}^m {L_{nk + jnq} L_{nk + jnq + np} } }}}  = \frac{1}{{5F_{mnq} }}\sum_{k = 1}^q {\frac{{( - 1)^{k - 1} }}{{\prod_{j = 0}^{m - 1} {L_{nk + jnq} L_{nk + jnq + np} } }}}\,. 
\]

\end{thm}

\subsection{Evaluation of other sums}

\begin{thm}\label{thm.jq353ph}
If $m$, $n$ and $q$ are positive integers, then 
\[
\sum_{k = 1}^\infty  {\frac{{( - 1)^{nk - 1} F_{2nk + mnq + 2} \prod_{j = 1}^{m - 1} {F_{nk + jnq + 1}^2 } }}{{\prod_{j = 0}^m {F_{nk + jnq} F_{nk + jnq + 2} } }}}  = \frac{q}{{F_{mnq} }} - \frac{1}{{F_{mnq} }}\sum_{k = 1}^q {\prod_{j = 0}^{m - 1} {\frac{{F_{nk + jnq + 1}^2 }}{{F_{nk + jnq} F_{nk + jnq + 2} }}} }\,, 
\]
so that
\begin{equation}
\sum_{k = 1}^\infty  {\frac{{F_{2nk + mnq + 2} \prod_{j = 1}^{m - 1} {F_{nk + jnq + 1}^2 } }}{{\prod_{j = 0}^m {F_{nk + jnq} F_{nk + jnq + 2} } }}}  = \frac{1}{{F_{mnq} }}\sum_{k = 1}^q {\prod_{j = 0}^{m - 1} {\frac{{F_{nk + jnq + 1}^2 }}{{F_{nk + jnq} F_{nk + jnq + 2} }}} }-\frac{q}{{F_{mnq} }},\quad\mbox{$n$ even}
\end{equation}
and
\begin{equation}
\sum_{k = 1}^\infty  {\frac{{( - 1)^{k - 1} F_{2nk + mnq + 2} \prod_{j = 1}^{m - 1} {F_{nk + jnq + 1}^2 } }}{{\prod_{j = 0}^m {F_{nk + jnq} F_{nk + jnq + 2} } }}}  = \frac{q}{{F_{mnq} }} - \frac{1}{{F_{mnq} }}\sum_{k = 1}^q {\prod_{j = 0}^{m - 1} {\frac{{F_{nk + jnq + 1}^2 }}{{F_{nk + jnq} F_{nk + jnq + 2} }}} },\quad\mbox{$n$ odd}\,.
\end{equation}

\end{thm}

\begin{thm}\label{thm.tpgmuma}
If $m$, $n$ and $q$ are positive integers such that $n$ is odd and $q$ is even, then 
\[
\sum_{k = 1}^\infty  {\left[ {\frac{{F_{2nk + mnq + 2} \prod_{j = 1}^{m - 1} {F_{nk + jnq + 1}^2 } }}{{\prod_{j = 0}^m {F_{nk + jnq} F_{nk + jnq + 2} } }}} \right]}  = \frac{1}{{F_{mnq} }}\sum_{k = 1}^q {\left[ {\prod_{j = 0}^{m - 1} {\frac{{( - 1)^k F_{nk + jnq + 1}^2 }}{{F_{nk + jnq} F_{nk + jnq + 2} }}} } \right]}\,. 
\]

\end{thm}
\subsubsection*{Proof of Theorem~\eqref{thm.jq353ph} and Theorem~\eqref{thm.tpgmuma}}
Dividing through the identity~\eqref{equ.yr8t5vs} by $F_{v+1}F_{v-1}F_{u+1}F_{u-1}$ and setting $u=nk+1$ and $v=nk+mnq+1$ we obtain the identity
\begin{equation}\label{equ.hro9n7d}
\frac{{( - 1)^{nk - 1} F_{mnq} F_{2nk + mnq + 2} }}{{F_{nk} F_{nk + 2} F_{nk + mnq} F_{nk + mnq + 2} }} = \frac{{F_{nk + mnq + 1}^2 }}{{F_{nk + mnq + 2} F_{nk + mnq} }} - \frac{{F_{nk + 1}^2 }}{{F_{nk + 2} F_{nk} }}\,.
\end{equation}
With $f(k)=F_{k+1}^2/(F_kF_{k+2})$ in Lemma~\ref{finall} and use of the identity~\eqref{equ.hro9n7d} we get the finite summation identity
\begin{equation}
\begin{split}
&F_{mnq} \sum_{k = 1}^N {\left[ {\frac{{( - 1)^{nk - 1} F_{2nk + mnq + 2} \prod_{j = 1}^{m - 1} {F_{nk + jnq + 1}^2 } }}{{\prod_{j = 0}^m {F_{nk + jnq} F_{nk + jnq + 2} } }}} \right]}\\ 
&\qquad\qquad = \sum_{k = 1}^q {\left[ {\prod_{j = 0}^{m - 1} {\frac{{F_{nk + nN + jnq + 1}^2 }}{{F_{nk + nN + jnq} F_{nk + nN + jnq + 2} }}} } \right]}\\
&\qquad\qquad\qquad  - \sum_{k = 1}^q {\left[ {\prod_{j = 0}^{m - 1} {\frac{{F_{nk + jnq + 1}^2 }}{{F_{nk + jnq} F_{nk + jnq + 2} }}} } \right]}\,, 
\end{split}
\end{equation}
from which Theorem~\eqref{thm.jq353ph} follows in the limit as $N$ approaches infinity. Theorem~\eqref{thm.tpgmuma} follows from the identity~\eqref{equ.hro9n7d} and taking \mbox{$f(k)=F_{k+1}^2/(F_kF_{k+2})$} in identity~\eqref{infallalt} with the upper sign.
\begin{thm}
\[
\sum_{k = 1}^\infty  {\frac{{F_{2k + 3} }}{{F_k^4 F_{k + 1}^3 F_{k + 2}^3 F_{k+3}^4 }}}  = \frac{1}{{128}},\quad\sum_{k = 1}^\infty  {\frac{{F_{2k + 3} }}{{L_k^4 L_{k + 1}^3 L_{k + 2}^3 L_{k+3}^4 }}}  = \frac{1}{{829440}}\,.
\]

\end{thm}
\begin{thm}
\[
\sum_{k = 1}^\infty  {\frac{{F_{3k + 1} F_{3k + 2} F_{6k + 3} }}{{F_{3k}^4 F_{3k + 3}^4 }}}  = \frac{1}{{128}},\quad\sum_{k = 1}^\infty  {\frac{{L_{3k + 1} L_{3k + 2} F_{6k + 3} }}{{L_{3k}^4 L_{3k + 3}^4 }}}  = \frac{1}{{10240}}\,.
\]

\end{thm}


\begin{thebibliography}{99}

\bibitem{adegoke}
K.~ADEGOKE (2017),
\newblock Some remarkable infinite product identities involving Fibonacci and Lucas numbers,
\newblock{\em arXiv:1702.08321}\\\url{https://arxiv.org/abs/1702.08321}.

\bibitem{brousseau1}
Bro.~A.~BROUSSEAU (1969),
\newblock Fibonacci-Lucas Infinite Series - Research Topic,
\newblock{\em The Fibonacci Quarterly} 7 (2):211--217.


\bibitem{brousseau2}
Bro.~A.~BROUSSEAU (1969),
\newblock Summation of infinite Fibonacci series,
\newblock{\em The Fibonacci Quarterly} 7 (2):143--168.


\bibitem{rabinowitz}
S.~RABINOWITZ (1999),
\newblock Algorithmic summation of reciprocals of products of fibonacci numbers,
\newblock{\em The Fibonacci Quarterly} 7 (2):122--127.

\bibitem{bruckman}
P.~S.~BRUCKMAN and I.~J.~GOOD (1976),
\newblock A generalization of a series of de Morgan, with applications of Fibonacci type,
\newblock{\em The Fibonacci Quarterly} 14 (3):193--196.

\bibitem{vajda}
S.~VAJDA (2008),
\newblock Fibonacci and Lucas numbers, and the Golden Section: Theory and Applications,
\newblock{\em Dover Press}

\bibitem{howard}
F.~T.~HOWARD (2003),
\newblock The Sum of the Squares of Two Generalized Fibonacci Numbers,
\newblock{\em The Fibonacci Quarterly} 41 (1):80--84.



\bibitem{melham16}
R.~MELHAM (2016), 
\newblock Finite reciprocal sums in which the denominator of the summand contains squares of generalized fibonacci numbers,
\newblock {\em Integers}, 16:1--11.

\bibitem{frontczak16}
R.~FRONTCZAK (2016), 
\newblock New results on reciprocal series related to fibonacci and lucas numbers with subscripts in arithmetic progression,
\newblock {\em International Journal of Contemporary Mathematical Sciences}, 11:509--516.

\bibitem{melham01}
R.~S.~MELHAM (2001),
\newblock Summation of reciprocals which involve products of terms from generalized Fibonacci sequences-Part II,
\newblock{\em The Fibonacci Quarterly} 39 (3):264--288.

\end{thebibliography}
\end{document}